\documentclass[11pt, reqno]{amsart}

\usepackage{amssymb,latexsym,amsmath,amsfonts}
\usepackage[usenames]{color} 
\usepackage{enumitem}
\usepackage{mathrsfs}
\usepackage{graphicx}
\usepackage{hyperref}
\usepackage{comment}

\definecolor{DPurple}{rgb}{0.46,0.2,0.69}

\numberwithin{equation}{section}

\allowdisplaybreaks

\theoremstyle{definition}
\newtheorem{definition}{Definition}[section]

\theoremstyle{remark}
\newtheorem{remark}[definition]{Remark}
\newtheorem{conjecture}[definition]{Conjecture}

\theoremstyle{plain}
\newtheorem{theorem}[definition]{Theorem}
\newtheorem{result}[definition]{Result}
\newtheorem{lemma}[definition]{Lemma}
\newtheorem{proposition}[definition]{Proposition}

\newtheorem{corollary}[definition]{Corollary}

\newcommand{\eps}{\varepsilon}
\newcommand{\K}{\kappa}
\newcommand{\exc}{\mathcal{E}}
\newcommand{\zt}{\zeta}

\newcommand{\F}{\boldsymbol{\rm F}}
\newcommand{\J}{\boldsymbol{\rm J}}
\newcommand{\filled}{\mathcal K}
\newcommand{\basin}{\mathcal A}
\newcommand{\lar}{\mathcal U}
\newcommand{\inv}{\mathcal A_{\infty}(\gen)}

\newcommand{\holo}{{\rm Hol}}

\newcommand\lead[1]{\Lambda(#1)}
\newcommand{\ord}{{\rm ord}}
\newcommand{\gen}{\mathcal G}
\newcommand{\fixed}{a} 
\newcommand{\arb}{y} 
\newcommand{\sub}{\nu}
\newcommand{\supp}{{\rm supp}}
\newcommand{\cpt}{{\rm cap}}
\newcommand{\critall}{\boldsymbol{C^*}}
\newcommand{\crit}{\boldsymbol{C}}
\newcommand{\wgen}{w_{\raisebox{-1.5pt}{$\scriptstyle{\gen}$}}}
\newcommand{\bad}{\mathfrak{B}}
\newcommand{\Q}{Q_{\gen}}
\newcommand{\weakST}{\xrightarrow{\text{weak${}^{\boldsymbol{*}}$}}}
\newcommand{\FS}{\omega_{FS}}
\newcommand{\corrC}{\Gamma_{\raisebox{-2.5pt}{$\scriptstyle{\!\mathscr{C}}$}}}
\newcommand{\equij}{\sigma_{\raisebox{3pt}{$\scriptstyle{\J(S)}$}}}

\newcommand{\sph}{\widehat{\mathbb{C}}}
\newcommand{\C}{\mathbb{C}} 

\newcommand{\Z}{\mathbb{Z}}
\newcommand{\N}{\mathbb{N}}

\begin{document}

\title[Dynamics of polynomial semigroups]{Dynamics of polynomial semigroups: measures, \\
potentials, and external fields}

\author{Mayuresh Londhe}
\address{Department of Mathematics, Indian Institute of Science, Bangalore 560012, India}
\email{mayureshl@iisc.ac.in}


\begin{abstract}
{In this paper, we give a description of a natural invariant measure associated with a finitely generated
polynomial semigroup (which we shall call the Dinh--Sibony measure)
in terms of potential theory. This requires the theory of logarithmic potentials
in the presence of an external field, which, in our case, is \emph{explicitly} determined by the choice
of a set of generators. Along the way, we establish the continuity
of the logarithmic potential for the Dinh--Sibony measure, which might be of independent interest.
We then use the $F$-functional of Mhaskar and Saff to discuss
bounds on the capacity and diameter of the Julia sets of such semigroups.}
\end{abstract}

\maketitle

\vspace{-0.25cm}
\section{Introduction and statement of main results}\label{S:intro}
A \emph{rational semigroup} is a subsemigroup of $\holo(\sph; \sph)$\,---\,the semigroup
with respect to composition of holomorphic self-maps of $\sph$\,---\,containing no constant maps
(where $\sph$ denotes the Riemann sphere). The investigation of such semigroups was
initiated by Hinkkanen and Martin in \cite{HinkMart:dsrf96}. Given a finitely generated rational
semigroup $S$ containing at least one
element of degree at least $2$, and a set of generators $\gen$, there happens to be a
dynamically meaningful probability measure $\mu_{\gen}$ associated with the pair $(S, \gen)$.
This paper is dedicated to the following question:
can one describe $\mu_{\gen}$, which is constructed purely dynamically, in terms of the theory
of logarithmic potentials? The motivation for this is that potential theory in $\C$ is
such a well-developed and deeply explored field that identifying $\mu_{\gen}$ in
potential-theoretic terms would reveal new information about the various invariant sets
associated with $S$.   
\smallskip

The measure $\mu_{\gen}$ is a measure that is preserved, in an appropriate sense, by a
holomorphic correspondence on $\sph$ associated with $(S, \gen)$. To
make precise what this means, we begin with

\begin{definition}\label{D:holCorr}
Let $X_1$ and $X_2$ be two compact, connected complex manifolds of dimension $k$. A
\emph{holomorphic correspondence} from $X_1$ to $X_2$ is a formal linear combination
of the form
\begin{equation}\label{E:stdForm}
 \Gamma = \sum\nolimits_{1\leq i\leq N}\!m_i\Gamma_i,
\end{equation}
where the $m_i$'s are positive integers and $\Gamma_1, \Gamma_2,\dots,\Gamma_N$ are distinct
irreducible complex-analytic subvarieties of $X_1\times X_2$ of pure dimension $k$ that satisfy the following
conditions:
\begin{enumerate}[label=(\roman*), leftmargin=27pt]
\item for each $\Gamma_i$ in \eqref{E:stdForm},
$\left.\pi_1\right|_{\Gamma_i}$ and $\left.\pi_2\right|_{\Gamma_i}$ are surjective;
\item for each $x\in X_1$ and $y\in X_2$, $\left(\pi_1^{-1}\{x\}\cap \Gamma_i\right)$ and
  $\left(\pi_2^{-1}\{y\}\cap \Gamma_i\right)$ are finite sets for each $i$
\end{enumerate}
(where $\pi_j$ is the projection onto $X_j, \ j=1,2$). 
\end{definition}
A holomorphic correspondence $\Gamma$ induces a set-valued function, which
we denote by $F_{\Gamma}$, 
\[
X_1\supseteq A \longmapsto \bigcup\nolimits_{1\leq i\leq N}\!\pi_2\left(\pi_1^{-1}(A)\cap \Gamma_i\right).
\]
We shall denote $F_{\Gamma}(\{x\})$ by $F_{\Gamma}(x)$.
If $X_1 = X_2 = X$ in the above definition then we say that $\Gamma$ is a \emph{holomorphic correspondence
\textbf{on} $X$}. Two holomorphic correspondences on $X$ can be composed with each other (see
Section~\ref{SS:DinSib}). This, and the map
$F_{\Gamma}$, introduce the dynamical element in the study of holomorphic correspondences.
\smallskip

Definition~\ref{D:holCorr} and the discussion immediately prior to it suggest the following natural

\begin{definition}\label{D:assoc}
Let $S$ be a finitely generated rational semigroup and let $\gen = \{g_1, g_2,\dots, g_N\}$ be a
set of generators of $S$, i.e., $S = \langle g_1, g_2,\dots, g_N\rangle$.
We call the following holomorphic correspondence
\begin{equation}\label{E:assoc}
\Gamma_{\gen} := \sum\nolimits_{1\leq i\leq N}{\rm graph}(g_i)
\end{equation}
the {\em holomorphic correspondence associated with $(S, \gen)$}.
\end{definition}

Now, $\mu_{\gen}$ arises from a very general construction by Dinh and
Sibony \cite{DinhSibony:dvtma06} applied to the holomorphic correspondence $\Gamma_{\gen}$.
We provide a little background. Let $X$ be a compact Riemann surface and $\Gamma$
a holomorphic correspondence on $X$. Let $d_1(\Gamma)$ be the generic number
of inverse images and $d_0(\Gamma)$ the generic number of
forward images under $\Gamma$,
both counted according to multiplicity (see Section~\ref{SS:DinSib}).
Dinh--Sibony show that regular Borel measures can be pulled back under $\Gamma$; for such
a measure $\mu$, let $F_{\Gamma}^*\mu$ denote its pull-back (see
\cite[Section~2.4]{DinhSibony:dvtma06} for details). The main results
of \cite{DinhSibony:dvtma06} applied to the latter set-up imply that when
$d_1(\Gamma) > d_0(\Gamma)$, there
exist a polar set $E\varsubsetneq X$
and a regular Borel probability measure $\mu_{\Gamma}$ such that
\begin{equation}\label{E:asymp1}
 \frac{1}{d_1(\Gamma)^n}F_{\Gamma^{\circ n}}^*(\delta_a) \weakST \mu_{\Gamma}
 \;\; \text{as $n\to \infty$,} \;\; 
 \forall a\in X\setminus E.
\end{equation}
The measure $\mu_{\gen}$ is the limit measure given by \eqref{E:asymp1} taking
$\Gamma = \Gamma_{\gen}$.
\smallskip

As $\mu_{\gen}$ is a special case of a construction in \cite{DinhSibony:dvtma06}, we
shall call it the \emph{Dinh--Sibony measure associated with $(S, \gen)$}. We must mention
that, under a further constraint on $S$\,---\,$S$ as above\,---\,which is reflected in his
choice of generating set $\gen$ of $S$, the measure $\mu_{\gen}$ was discovered by Boyd
\cite{boyd:imfgrs99}. Also see \cite{sumi:spmrfgrs00} by Sumi for another approach to
associating dynamically interesting measures with rational semigroups. Boyd's construction is
not based on the formalism of correspondences. Our main theorems, however,
do not rely principally on Boyd's construction, nor do they rely on his methods. We shall not
dwell on the reasons for this, but the interested reader is referred to
\cite[Remark~4.1]{BhaSri:hcrfgrs17} and to the fact that the semigroups that we
shall consider are allowed to have degree-one elements. The semigroups we shall consider
are described by the following

\begin{definition}\label{D:poly}
A rational semigroup $S$ is called a \emph{polynomial semigroup} if all its elements are polynomials,
any degree-one element of $S$ has an attracting fixed point at $\infty$, and $S$ contains at least one
element of degree at least $2$.
\end{definition}

\begin{remark}
Unlike what its name suggests, a polynomial semigroup cannot contain arbitrary degree-one
elements. Yet, we choose the latter name for the semigroups considered here because we
want the terminology to evoke Brolin's result \cite[Theorem~16.1]{brolin:isuirm65} on an
invariant measure associated with a polynomial $P$ of degree at least 2. That $P$ has an
attracting fixed point at $\infty$ is a crucial part of Brolin's proof. This is what motivates
our condition on degree-one elements in Definition~\ref{D:poly}.
In fact, we shall see that our Theorem~\ref{T:main} subsumes Brolin's theorem.
\end{remark}
 
We first show that the conditions defining a \emph{finitely} generated
polynomial semigroup $S$ imply something
interesting about its generators. A set of generators 
$\gen=\{ g_1, g_2,\dots,g_N\}$ of $S$ is called $\emph{minimal}$ 
if no function $g_i$ can be expressed as a composition involving the remaining generators.
The existence of such sets is clear, but more can be said: 

\begin{proposition}\label{P:unique}
Every finitely generated polynomial semigroup has a unique minimal generating set.
\end{proposition}

This proposition is important because, given a finitely generated polynomial semigroup
$S$, it identifies a set of generators of $S$ that is, in a precise sense, \textbf{canonical.}
We shall denote the unique minimal generating set of $S$ as $\gen_S$.
\smallskip

In \cite[Theorem~16.1]{brolin:isuirm65}, Brolin constructed an invariant measure associated with a
polynomial $P$ of degree at least 2 (which turns out to be precisely the Dinh--Sibony measure
associated with $(S, \gen~\!\!=~\!\!\{P\})$) and showed it to be the equilibrium measure of the Julia set of
$P$. 
This result \textbf{cannot} extend
naively to finitely generated polynomial semigroups $S$ with more than one generator because:
\renewcommand{\theenumi}{{\emph{\roman{enumi}}}}
\begin{enumerate}[leftmargin=27pt]
\item\label{I:interior} It is easy to construct finitely generated polynomial semigroups $S$ whose
Julia sets $\J(S)$ (see Section~\ref{S:def} for a definition) have non-empty interiors. See, for instance,
\cite[Example~1]{HinkMart:dsrf96}.

\item\label{I:support} There exist finitely generated polynomial semigroups $S$ as in
\eqref{I:interior} and a choice of generating set
$\gen$ such that $\supp(\mu_{\gen}) = \J(S)$. See \cite[Theorem~1]{boyd:imfgrs99}.
\end{enumerate}
Now, if for a semigroup of the above kind, and a choice of generating set as in \eqref{I:support},
the measure $\mu_{\gen}$ were the equilibrium measure of $\J(S)$ then it would have to be
supported on the exterior boundary of $\J(S)$, which would contradict $\eqref{I:support}$. This
is the fundamental problem one must understand in order to answer the question posed at the beginning
of this section.
\smallskip

The latter problem is solved by turning to the theory of logarithmic potentials in
the presence of an external field. Roughly speaking, an \emph{equilibrium measure associated with an
external field} gives the distribution of a unit charge on a conductor, in the presence of an
external electrostatic potential, that minimizes energy (the classical equilibrium measure
gives the latter distribution in the absence of any external field). To make mathematical sense,
this electrostatic potential
must satisfy certain admissibility conditions and, in the mathematical literature, is called an
\emph{external field}: see Section~\ref{S:potential} for definitions and
\cite[Chapter~I]{SaffTotik:lpwef97} by Saff and Totik for details. Our first theorem says,
in essence, that, given a finitely generated polynomial semigroup $S$ and a finite set of
generators $\gen$, the measure $\mu_{\gen}$ is the equilibrium measure associated with
an external field that is described \textbf{explicitly} by $\gen$. Now we introduce this external
field. For $(S, \gen)$ as above, we define the \emph{dynamical Green's function associated
with $(S, \gen)$} to be the upper semicontinuous regularization of
\[
G_{\gen}(z):=\limsup_{n \to \infty} \frac{1}{d_1(\Gamma_{\gen})^n} 
\log \Bigg(\displaystyle \prod_{l(g)=n} {|g(z)-a|}\Bigg),
\]
where $a$ is arbitary element outside a certain polar set (see Section~\ref{S:def} for the
meaning of the above product).
Let us denote the upper semicontinuous
regularization of $G_{\gen}$\,---\,see Section~\ref{S:main} for a definition\,---\,by $G^*_{\gen}$.
That $G^*_{\gen}$ does not depend on $a$ as above is a consequence of 
\eqref{E:asymp1} with $\Gamma = \Gamma_{\gen}$\,---\,we shall say more
about this; see Remark~\ref{Rem:independence}. 
The external field that is relevant to our problem is $G^*_{\gen}$ restricted to $\J(S)$.
\smallskip

A point $z \in \sph$ is called $\emph{exceptional}$ if the set
\[
O^{-}(z):=\{ x \in \sph : g(x)=z \text{ for some } g \in S\}
\] 
is finite. We denote the set of exceptional points by $\exc(S)$. It is well known that,
for a rational semigroup $S$, $\sharp(\exc(S))\leq 2$. Note that for a polynomial semigroup $S$,
$\infty \in \exc(S)$, and hence it has at most one exceptional point in $\C$.
Now we are in the position to state

\begin{theorem}\label{T:main}
Let $S$ be a finitely generated polynomial semigroup. Define
\[
\begin{split}
\critall[S] &:=\{c \in \C:  g'(c)=0 \text{ for some } g\in \gen_S\}, \\
\crit[S] &:= \{c \in \J(S):  g'(c)=0 \text{ for some } g\in \gen_S\}.
\end{split}
\]
Suppose $S$ has the property that if $\sharp\critall[S] = 1$ then 
$\crit[S] \cap \exc(S)=\emptyset$.
Then, for \emph{any} finite set of generators $\gen$ of $S$, the Dinh--Sibony measure 
$\mu_{\gen}$ is the equilibrium measure associated with the external field $G^*_{\gen}|_{\J(S)}$.
\end{theorem}

\begin{remark}\label{Rem:broad}
The condition stated in Theorem~\ref{T:main} is very mild. Finitely generated polynomial semigroups
that do \textbf{not} satisfy the condition stated are very exceptional\,---\,see Section~\ref{SS:excep}
for a precise discussion. If, for a finitely generated polynomial semigroup $S$, $\sharp\critall[S] = 1$ and
$\crit[S] \cap \exc(S)\neq \emptyset$, then it is unclear whether the logarithmic potential of the
measures $\mu_{\gen}$ (for any generating set $\gen$ of $S$) is continuous. That being said, the
behaviour of these potentials is not intractable. But since a completely different analysis would be
required to study these exceptional semigroups, we focus here on the semigroups addressed by 
Theorem~\ref{T:main}.
\end{remark}

If $P$ is a polynomial of degree at least 2 and we consider 
the iterative dynamics of $P$ (i.e., $\gen =\{P\}$ in our notation), it follows from
\cite{brolin:isuirm65} that $G_{\gen}$ restricted to
the unbounded component of the Fatou set of $P$ is the Green's function of the latter
with pole at $\infty$, and
the external field is identically $0$ in this case. This is the framework of classical potential
theory. From our remark above on the equilibrium measure (see Section~\ref{S:potential} for
a precise statement), and as $\mu_{\gen}$ here is the equilibrium measure
of the Julia set, we see that Brolin's theorem is subsumed by Theorem~\ref{T:main}.
\smallskip

We provide a very short sketch of the proof of Theorem~\ref{T:main} to point out some
features of it that are novel (in what follows, $D(z, r)$ denotes
the open disc with centre $z\in \C$ and radius $r$): 
\begin{itemize}[leftmargin=14pt]
\item We establish that the logarithmic potential of $\mu_{\gen}$ (let us denote it by 
$U^{\mu_{\gen}}$) is continuous. This is 
Theorem~\ref{T:conti}\,---\,which may be of independent interest.

\item We show that there is a constant $\alpha > 0$ such that for each
$z\in \C$, $\mu_{\gen}(D(z,r))\lesssim r^\alpha$ 
when $r > 0$ is sufficiently small.
The continuity of $U^{\mu_{\gen}}$ follows from this using a
characterization of the continuity of the logarithmic potential by Arsove
\cite[Theorem~1]{arsove:cplmd60}.

\item Using the continuity of $U^{\mu_{\gen}}$ we are able to show a very strong
relation between $U^{\mu_{\gen}}$ and the external field in Theorem~\ref{T:main},
from which this result follows.
\end{itemize}
We should mention that the proof of the power bound on $\mu_{\gen}(D(z,r))$ is
inspired by an argument by Bharali and Sridharan \cite[Section~5]{BhaSri:hcrfgrs17}.
However, their argument addresses only the case $\crit[S] = \emptyset$. We make a
careful analysis of the \textbf{local} orders at critical points to show
that the power bound on $\mu_{\gen}(D(z,r))$ can be
obtained even when $\crit[S]\neq \emptyset$. This is our 
Proposition~\ref{P:card}.%
\smallskip

With Theorem~\ref{T:main} at hand, one expects that a growing understanding of the external
field would lead to new information about dynamically interesting objects associated with $S$ as
in Theorem~\ref{T:main}. In this work, we focus on the (logarithmic) capacity of $\J(S)$. Our
next result shows why it is fundamentally \textbf{hard} to compute the capacity of $\J(S)$ for $S$
with more than one generator. To be specific: for the latter $S$, the analogue of the Robin constant for
$\J(S~\!\!=~\!\!\langle P\rangle)$, $P$ a polynomial, is the \emph{modified
Robin constant} (see Section~\ref{S:potential}), $F_{\gen}$, for the external field given by
Theorem~\ref{T:main}. Just like the Robin constant for
$\J(S~\!\!=~\!\!\langle P\rangle)$, $F_{\gen}$ is not hard to compute.
But the relationship between $F_{\gen}$ and capacity is badly
vitiated by the external field $G^*_{\gen}|_{\J(S)} =: Q_{\gen}$. One may ask: when
is $Q_{\gen}\not\equiv 0$? Theorem~\ref{T:bound} below provides a sufficient condition under
which $Q_{\gen}\not\equiv 0$. Actually, we conjecture that the condition~$(ii)$ in
Theorem~\ref{T:bound}-\ref{I:ext_field} is superfluous and that condition~$(i)$ is
necessary and sufficient for $Q_{\gen}\not\equiv 0$: refer to
Conjecture~\ref{Conj:capac} for a precise statement. We shall prove that:
\begin{itemize}[leftmargin=14pt]
\item The condition~$(i)$ in Theorem~\ref{T:bound}-\ref{I:ext_field} is a necessary
condition for $Q_{\gen}\not\equiv 0$ for some finite generating set $\gen$,
which is part of our evidence for Conjecture~\ref{Conj:capac}.

\item Under the conditions on $S$ alluded to, $F_{\gen}$ is always greater than
the Robin constant of $\J(S)$ (for any finite generating set $\gen$), which gives a
strict lower bound on the capacity of $\J(S)$.
\end{itemize}
\noindent{Some notation: $\lead{g}$ will denote the coefficient of the highest-degree
term of a polynomial $g$.}

\begin{theorem}\label{T:bound}
Let $S$ be a finitely generated polynomial semigroup having the properties stated
in Theorem~\ref{T:main}. For a set of generators 
$\gen=\{ g_1, g_2,\dots,g_N\}$ of $S$, let $Q_{\gen}$ denote the external field 
associated with $(S, \gen)$ given by Theorem~\ref{T:main}.
\begin{enumerate}[label=$(\alph*)$]
\item\label{I:ext_field} Assume that for some $z_0\in \J(S)$:
$(i)$~its orbit under $S$, $O(z_0)$, is unbounded, and $(ii)$~$O(z_0)$ is
not dense in $\C$. Then $Q_{\gen}\not\equiv 0$ for any finite set of generators $\gen$ of $S$.
\vspace{1mm}

\item\label{I:nec_cond} If $Q_{\gen}\not\equiv 0$ for some finite set of generators $\gen$
then there exists a point $z_0\in \J(S)$ such that $O(z_0)$ is unbounded.
\end{enumerate}
Moreover, if $S$ satisfies conditions~$(i)$ and~$(ii)$ and 
each element of $S$ is of degree at least 2 then
\begin{equation}\label{E:bound_cap}
\cpt(\J(S)) > \exp(-F_{\gen}),
\end{equation}
for any set of generators $\gen$ as above. Here
$F_{\gen}$ is the modified Robin constant for $Q_{\gen}$, and equals
$(D-N)^{-1}\log |\lead{g_1}\lead{g_2}\dots\lead{g_N}|$, where
$D:=\deg(g_1)+\deg(g_2)+\dots +\deg(g_N)$.
\end{theorem}

As discussed above, we have some evidence to propose the following 

\begin{conjecture}\label{Conj:capac}
Let $S$ be a finitely generated polynomial semigroup. For a finite set of generators 
$\gen$ of $S$, let $Q_{\gen}$ denote the external field 
associated with $(S, \gen)$ given by Theorem~\ref{T:main}. Then, the following are equivalent:
\begin{enumerate}[label=$(\alph*)$]
\item For some point $z_0\in \J(S)$, its orbit under $S$ is unbounded.

\item $Q_{\gen}\not\equiv 0$ for some finite set of generators $\gen$ of $S$.

\item $Q_{\gen}\not\equiv 0$ for every finite set of generators $\gen$ of $S$.
\end{enumerate}
\end{conjecture}

We conclude this section with another indication of future research. To this end, we mention a recent work
by Dinh \emph{et al.} \cite{DinhKaufWu:prmdpv20}, with which this work shares some common features.
In \cite{DinhKaufWu:prmdpv20}, the authors revisit the problem in random matrix theory of determining 
the asymptotic behaviour of the random products $s_n\cdots s_1$, $n=1,2,3,\dots,$ where $s_1, s_2, s_3, \dots$
are sampled independently and identically from ${\rm SL}_2(\C)$ relative to a non-elementary probability measure $\mu$
on ${\rm SL}_2(\C)$.
The approach
in \cite{DinhKaufWu:prmdpv20} involves interpreting each $s_n\in\rm SL_2(\C)$ as an element in ${\rm Aut}(\sph)$.
The problem considered leads to the study of the dynamical properties of a ``generalized correspondence''\,---\,in the sense
that if $\supp(\mu)$ were a \textbf{finite} set then the latter object would be a correspondence of the form
\eqref{E:stdForm}. This suggests
a class of problems on $S$\,---\,$S$ as in Theorems~\ref{T:main} and~\ref{T:bound}\,---\,wherein
one considers \emph{random} compositions in $S$ by endowing $S$ with a probability measure $\mu$ analogous to
those considered in \cite{DinhKaufWu:prmdpv20}. The analytical methods in \cite{DinhKaufWu:prmdpv20} 
suggest, for instance, a way to describe the asymptotic behaviour of random compositions in the
higher-degree setting.

\section{Classical notions on rational semigroups}\label{S:def}
For a rational semigroup $S$, the \emph{Fatou set} of $S$\,---\,which we denote by 
$\F(S)$\,---\,is the largest open subset of $\sph$ on which functions in $S$ form a 
normal family. The \emph{Julia set} of $S$, $\J(S)$, is the complement of 
$\F(S)$: i.e., $\J(S):=\sph \setminus \F(S)$. 
If $S$ is generated by a set
$\gen$ then we write $S=\langle f: f \in \gen \rangle$.
If $S$ is generated by a single rational
function $f$ then we abbreviate $\F(\langle f \rangle)$ and
$\J(\langle f \rangle)$ to  $\F(f)$ and $\J(f)$, respectively.
\smallskip
 
If a rational semigroup $S$ contains a function of degree at least 2 then we 
can say more about $\J(S)$.

\begin{result}[Hinkkanen--Martin, \cite{HinkMart:dsrf96}]\label{R:HM}
Let $S$ be a rational semigroup containing at least one function of degree at least 2. Then
the following hold:
\begin{enumerate}[label=$(\alph*)$]
\item The set of all repelling fixed points of all the elements of $S$ is dense in $\J(S)$.  
\item The Julia set of $S$ satisfies 
       \[\J(S)=\overline{\bigcup_{f\in S} \J(f)}.\]
\end{enumerate}
\end{result}

Let $S$ be a rational semigroup and let $\gen$ 
be a set of generators (or a generating set) of $S$. A \emph{word} will refer to any composition 
$f_{i_n} \circ \dots \circ f_{i_1}$, $n\in \Z_+$, where 
$f_{i_1}$, $\dots$, $f_{i_n} \in \gen$.
We shall call $n$ the \emph{length} of the 
\textbf{word} $f_{i_n} \circ \dots \circ f_{i_1}$.
For $f \in S$, the expression $l(f)=n$ is the shorthand for the following implication:
\[
l(f)=n \implies  \exists \ f_{i_1}, \dots, f_{i_n} \in \gen \text{ \  such that \ }
f=f_{i_n} \circ \dots \circ f_{i_1}.
\]
In the above expression, the word
$f_{i_n} \circ \dots \circ f_{i_1}$ is called as a \emph{representation} of $f$.
Note that for $f \in S$ there could be more than one representation of $f$. 

\begin{definition}\label{D:filled}
Let $S$ be a polynomial semigroup. The $\emph{filled-in Julia set}$ of $S$, denoted by 
$\filled(S)$, is the set
\[
\filled(S):=\{ z \in \C: O(z) \text{ has finite limit point}\},
\]
where $O(z):= \{f(z): f \in S\}$ denotes the orbit of $z \in \C$. We shall call the complement of 
$\filled(S)$, denoted by $\basin(S)=\sph\setminus \filled(S)$, the \emph{basin of attraction} 
of $\infty$ for $S$. 
\end{definition}

A couple of observations are in order. First: we recall that the expression $O(z)$
\emph{has finite limit point} means that there exists a
sequence $\{h_n\}$ in $S$ consisting of \textbf{distinct} elements of $S$ such that 
$\{h_n(z)\}$ converges to some point in $\C$. Second: it is not
immediate from the above definition why $\basin(S)$ is
called a ``basin of attraction''. The terminology is made clearer by
Lemma~\ref{L:compact} below, which provides an alternative description
of $\basin(S)$.
\smallskip

Julia sets of finitely generated rational semigroups have an interesting property
that we will need in our proof of Theorem~\ref{T:bound}. We first state this property
and then the pertinent result.
  
\begin{definition}\label{D:unif_perf}
Let $\Sigma$ be a closed subset of $\sph$. We say that $\Sigma$ is \emph{uniformly perfect}
if $\Sigma$ contains at least two points and there exists a number $M\in (0, \infty)$ such that
for any conformal annulus $\mathcal{A}\subset \sph$ that separates $\Sigma$ (which
means:
\begin{itemize}[leftmargin=14pt]
\item $\mathcal{A}\cap \Sigma = \emptyset$, and
\item $\Sigma$ intersects both the connected components of $\sph\setminus \mathcal{A}$)
\end{itemize}
the modulus of $\mathcal{A}$ is at most $M$.
\end{definition}

\begin{remark}\label{Rem:unif_perf_planar}
Let $\Sigma$ be a non-empty closed subset of $\C$. It is a classical fact\,---\,which follows
from the work of Pommerenke \cite{pommerenke:upsFg84}\,---\,that uniform perfectness is equivalent
to the following property: there exists a number
$c\in (0,1)$ such that for each $z\in \Sigma$ and each $r\in (0, {\rm diam}(\Sigma))$,
\[
 \Sigma\cap \{w\in \C : cr\leq |w-z|\leq r\}\neq \emptyset.
\]
Here, ${\rm diam}(\Sigma)$ denotes the diameter of $\Sigma$, and takes values
in $(0, \infty]$.
\end{remark}

\begin{result}[paraphrasing of {\cite[Theorem~3.1]{stankewitz:upsrsKg2000}}]\label{R:uniper}
Let $S$ be a finitely generated rational semigroup such that $J(S)$ has at least three points.
Then $\J(S)$ is uniformly perfect.
\end{result}

\section{Structural lemmas for polynomial semigroups}\label{S:structural}
In this section, we present certain lemmas about polynomial semigroups that will be of
use in later sections in this work. Recall the definition of a polynomial semigroup from
Section~\ref{S:intro}: such a semigroup
contains at least one polynomial of degree at least 2.
For a polynomial semigroup $S$, if $\gen$ is a set of generators then $\gen$ must contain 
at least one element of degree at least 2. This fact will be used implicitly, with no further 
comment, in the rest of this paper.
\smallskip

If $g (z)=az+b$ is a polynomial such that $|a|>1$ then $\J(g)=\{{b}/{(1-a)}\}$.
If $p$ is a polynomial of degree at least 2 then $\infty$ is a superattracting fixed point and
$\infty \in \F(g)$. This information is not enough for concluding whether $\infty \in \F(S)$.
But if we have a \textbf{finitely generated} polynomial semigroup then $\infty \in \F(S)$. 
To prove this result, we need a lemma, which is interesting in its own right.
\smallskip

\begin{lemma}\label{L:repre}
Let $S$ be a finitely generated polynomial semigroup and let $\gen=\{ g_1, g_2,\dots,g_N\}$ 
be a set of generators of $S$. For every $g \in S$, there are at most finitely many
representations of $g$ in terms of elements of $\gen$.
\end{lemma}

\begin{proof}
Fix $g \in S$. There is a finite number, say $m$, such that in any representation of $g$,
the number of elements of $\gen$ with degree at least 2 does not exceed $m$.
Thus, if every element of $\gen$ is of degree at least 2 then the result follows. We now consider
the case when there are degree one elements in $\gen$. Then, let
\[
\begin{split}
\Lambda_1 &:=\min \{|\lead{h}|:h \in \gen, \  \deg(h)=1\},\\
\Lambda_2 &:=\min \{|\lead{h}|:h \in \gen, \  \deg(h)\geq 2\},\\
d &:=\max\{\deg(h): h \in \gen\}.
\end{split}
\]
Note that $\Lambda_1 >1$ and $\Lambda_2 >0$.
\smallskip

Consider a word $g_{i_n} \circ \dots \circ g_{i_1}$ with at most 
$m$ elements with degree at least 2 and $n>m$. Note that
\[
\lead{g_{i_n} \circ \dots \circ g_{i_1}}=\lead{g_{i_n}}{\lead{g_{i_{n-1}}}}^{d_{i_n}} 
\dots \lead{g_{i_1}}^{d_{i_2} \dots d_{i_{n}}},
\]
where $d_{i_j}= \deg (g_{i_j})$. It is easy to see that if $\Lambda_2 <1$ then
\[
|\lead{g_{i_n} \circ \dots \circ g_{i_1}}| 
\geq {\Lambda_1}^{n-m}{\Lambda_2}^{1+d+\dots+d^{m-1}}.
\]
Similarly, if $\Lambda_2 \geq 1$ then
\[
|\lead{g_{i_n} \circ \dots \circ g_{i_1}}| \geq {\Lambda_1}^{n-m}.
\]
In both cases, the right hand sides of the above inequalities approach $\infty$ as 
$n\rightarrow \infty$. So there exists $n_0\in \Z_+$ such that
(here $g \in S$ is as fixed above)
\begin{equation}\label{E:repre}
|\lead{g_{i_n} \circ \dots \circ g_{i_1}}| > |\lead{g}| \quad \forall n \geq n_0.
\end{equation}
If $g_{i_n} \circ \dots \circ g_{i_1}$ is a representation of $g$
then $\lead{g_{i_n} \circ \dots \circ g_{i_1}} = \lead{g}.$ 
Therefore, by \eqref{E:repre}, $n <n_0$. Since $\gen$ is finite, there are at most 
finitely many words with length $\leq n_0$.
Thus there are at most finitely many representations of $g$ in terms of elements of 
the generating set $\gen$.
\end{proof}

Let $x\in \basin(S)$. By definition, either $x=\infty$ or $O(x)$ does not have finite limit points
in the sense of the explanation following Definition~\ref{D:filled}. 
Thus, given $r>0$, $\{g\in S: |g(x)| \leq r\}$ is a finite set. Owing to Lemma~\ref{L:repre}, 
for $g \in S$, there at most finitely many representations. 
Thus there exists $n_r(x) \in \Z_+$ such that if $l(g)\geq n_r(x)$ then
$|g(x)| > r$.

\begin{lemma}\label{L:compact}
Let $S$ be a finitely generated polynomial semigroup and let $\gen=\{ g_1, g_2,\dots,g_N\}$ 
be a set of generators of $S$. Then:
\begin{enumerate}[label=$(\alph*)$]
\item\label{I:large_R} There exist constants $M>1$ and $R>0$ such that, for each
$i \in \{1, \dots, N\}$, $|g_i(z)|> M|z|$ whenever $|z|>R$.
\item If we set $\lar:=\{|z|>R\} \cup \{\infty\}$ and define
\[
\inv:=\bigcup_{n=1}^{\infty}\bigg( \bigcap_{l(g)=n}g^{-1}(\lar) \bigg),
\]
then $\inv$ does not depend on the choice $R > 0$ in the definition of
$\lar$ for any $R>0$ and $M>1$ for which the conclusion of \ref{I:large_R} holds true.
\item $\basin(S)=\inv$ and, therefore, $\filled(S)$ is a compact subset of $\C$.
\end{enumerate}
\end{lemma}

\noindent{The proof of this lemma is routine. 
Therefore, we shall not provide a proof but,
instead, make a few explanatory remarks. The conclusion of \ref{I:large_R} is 
a consequence of the fact that, by definition, $\infty$ is an attracting
fixed point for each $g_i$, $i\in \{1, \dots, N\}$. The equality of
$\basin(S)$ and $\inv$ relies on the fact\,---\,observed just prior to the stament of
Lemma~\ref{L:compact}\,---\,that whenever $x\in \basin(S)$, there exists 
$n_{R}(x) \in \Z_+$ such that if $l(g)\geq n_{R}(x)$ then $g(x) \in \lar$.} 
\smallskip

Since, by definition of $\filled(S)$, $\J(g) \subset \filled(S)$ for every $g \in S$, it follows from Result~\ref{R:HM}
that $\J(S) \subset \filled(S)$. Thus $\J(S)$ is also compact in $\C$ and thus $\infty \in \F(S)$. 
In general, if $S$ is a polynomial semigroup (not necessarily finitely generated) 
then above lemma is not true\,---\,just consider $S=\langle z^2/n:n \in \Z_+\rangle$.
\smallskip

If $g$ is a polynomial of degree at least 2 then it is known that $\J(g)=\partial \filled(g)$. 
For a polynomial $g (z)=az+b$ with $|a|>1$, we can see that 
$\J(g)=\filled(g)=\partial \filled(g)$. It is now easy to see that
\[
\J(S) \subset \overline{\bigcup_{g\in S} \filled(g)}\subset \filled(S).
\]
In the above expression, inclusions can be strict but if every element in $S$ 
is of degree at least 2 then we have

\begin{result}[Boyd, \cite{boyd:ibaipsft04}]\label{R:boyd04}
Let S be a finitely generated polynomial semigroup where each element of $S$ has degree
at least 2. Then the unbounded components of $\basin(S)$ and $\F (S)$ are same.
\end{result}

\begin{remark}
In the version of the above result that Boyd establishes, see \cite[Theorem 4.1]{boyd:ibaipsft04},
he considers a class of semigroups that he calls \emph{polynomial semigroups of finite type}\,---\,see
\cite[Definition 3.1]{boyd:ibaipsft04}.
It follows from Lemma~\ref{L:compact} that a finitely generated polynomial semigroup where each 
element of $S$ has degree at least 2 is of finite type, which gives us Result~\ref{R:boyd04}.
\end{remark}

We now consider two results that have been referenced in
Section~\ref{S:intro}. First, we provide

\begin{proof}[The proof of Proposition~\ref{P:unique}]
Let $S$ be a finitely generated polynomial semigroup. Consider a finite generating 
set for $S$. Now remove one-by-one the generators that can be expressed as 
compositions of the other generators. Eventually we will end up having a minimal generating 
set. So, $S$ has at least one minimal generating set. Let
\[
S= \langle g_1,g_2,\dots,g_N \rangle = \langle h_1,h_2,\dots,h_{N'} \rangle,
\]
where each generating set is minimal. Now, if such sets are not unique then we may assume
$\{ g_1,g_2,\dots,g_N \} \neq \{ h_1,h_2,\dots,h_{N'} \}.$
Without loss of generality we may assume that  $g_1 \notin \{ h_1,h_2,\dots,h_{N'} \}$. 
But as $\{ h_1, h_2,\dots,h_{N'}\}$ is a generating set,
\[
g_1=h_{i_r}  \circ \dots \circ h_{i_1}
\]
for some $i_1, \dots, i_r \in \{1, \dots, N'\}$ and $r \geq 2$. Now, since $h_{i_j} \in 
S= \langle g_1,g_2,\dots,g_N \rangle $, for every $1 \leq j \leq r$, we get
\begin{equation}\label{E:compo}
g_1=g_{i_n} \circ \dots \circ g_{i_1}
\end{equation}
for some $i_1, \dots, i_n \in \{1, \dots, N\}$ and $n \geq 2$.
If $g_1 \neq g_{i_j}$ for every $1 \leq j \leq n$ then we would have a contradiction of the 
minimality of the generating set $\{ g_1,g_2,\dots,g_N \}$. Hence,
$g_1 = g_{i_{j^*}}$ for some $1 \leq j^* \leq n$.
\smallskip

If $\deg (g_1) =1$ then by \eqref{E:compo}, 
$\deg (g_{i_1}) = \dots =\deg (g_{i_n})=1$.
Since $|\lead{g}|>1$ for every $g\in S$ such that $\deg (g) =1$,
\[
\begin{split}
 |\lead{g_{i_n} \circ \dots \circ g_{i_1}}|
 &=|\lead{g_{i_n}} \dots \lead{ g_{i_1}}|\\
 &>|\lead{g_{i_{j^*}}}| =|\lead{g_1}|,
\end{split}
\]
which contradicts \eqref{E:compo}.
Now, if $\deg (g_1) \geq 2$ then by \eqref{E:compo},
\[
g_1=p_2 \circ g_{i_{j^*}} \circ p_1,
\]
where $p_1=g_{i_{j^*-1}} \circ \dots \circ g_{i_1}$ and 
$p_2=g_{i_n} \circ \dots \circ g_{i_{j^*+1}}$. Therefore, 
$\deg (p_1)=\deg (p_2)=1$ and $|\lead{p_1}|=|\lead{p_2}| > 1$.
We compute:
\[
\begin{split}
|\lead{g_{i_n} \circ \dots \circ g_{i_1}}| 
& = |\lead{p_2 \circ g_{i_{j^*}} \circ p_1}|\\
& =|\lead{p_2} \lead{g_{i_{j^*}}}{\lead{p_1}}^{\deg (g_{i_{j^*}})}| \\
& > |\lead{g_{i_{j^*}}}| = |\lead{g_1}|,
\end{split}
\]
which again contradicts \eqref{E:compo}. Thus $S$ has a unique minimal generating set.
\end{proof}

We now state and prove a simple lemma that is essential to the proof of
Theorem~\ref{T:main}. To do so, we need some notation. If $A$ is a 
finite set (respectively, a finite list) whose elements (respectively, terms) are
the non-constant polynomials $g_1, g_2,\dots, g_N$,
then we set
\begin{align}
\critall(A) &:=\{c \in \C:  g_i'(c)=0 \text{ for some } i\in \{1,\dots,N\}\}, \label{E:critall} \\
\crit(A) &:= \{c \in \J(S):  g_i'(c)=0 \text{ for some } i\in \{1,\dots,N\}\}. \label{E:crit}
\end{align}
With this notation, we have:

\begin{lemma}\label{L:crit_condtn}
Let $S$ be a finitely generated polynomial semigroup. The condition
\[
\sharp\critall(\gen_S)=1\,\Rightarrow\,\crit(\gen_S)\cap \exc(S) = \emptyset
\]
holds true if and only if the condition
\[
\sharp\critall(\gen)=1\,\Rightarrow\,\crit(\gen)\cap \exc(S) = \emptyset
\]
holds true for any set of generators $\gen$ of $S$.
\end{lemma}

\begin{proof}
The ``if'' part of the above assertion is obvious.
Now suppose that the condition $\big(\sharp\critall(\gen_S)=1 \Rightarrow \crit(\gen_S)\cap \exc(S) = \emptyset\big)$ 
holds true. If, for some set of generators $\gen$, we have
$\sharp\critall(\gen)=1$, then $\critall(\gen_S)=\critall(\gen)$. This is because $\gen_S\subseteq \gen$ and
$\critall(\gen_S)\neq \emptyset$.
Thus $\crit(\gen_S)=\crit(\gen)$, whence
$\big(\sharp\critall(\gen)=1\Rightarrow \crit(\gen)\cap \exc(S) = \emptyset\big)$ holds true.
\end{proof}

We conclude this section with a discussion on the type of polynomial semigroups
\textbf{excluded} by the condition in Theorem~\ref{T:main}.

\subsection{On the exceptional semigroups of Remark~\ref{Rem:broad}}\label{SS:excep}
Let $S$ be a finitely generated polynomial semigroup that does not satisfy the condition in
Theorem~\ref{T:main}. Let us write $\gen_S = \{g_1, g_2,\dots, g_N\}$. As
$\sharp\critall[S] = 1$ and $\crit[S] \cap \exc(S)\neq \emptyset$, there exists a point
$a\in \C$ such that $\critall[S] = \crit[S] = \{a\}$. From this, and the fact\,---\,evident
from the definition\,---\,that $g^{-1}(\exc(S))\subseteq \exc(S)$ for every $g\in S$, we
see that
\[
g_j(z) = B_j(z-a)^{n_j} + a \quad\text{whenever $\deg(g_j)\geq 2$},
\]
for some constant $B_j\in \C \setminus \{0\}$ and some $n_j\in \Z_+\setminus\{1\}$.
At this stage, we record the following%
\smallskip

\noindent{\textbf{Fact.} \emph{If $\mathscr{S}$ is a finitely generated rational semigroup each of whose
elements has degree at least 2 then $\exc(\mathscr{S})\subset \F(\mathscr{S})$.}}
\smallskip

\noindent{whose proof is routine. Since $\crit[S] \cap \exc(S)\neq \emptyset$, it follows
from the above fact that $\gen_S$ must contain degree-one elements. Once again, as
$g^{-1}(\exc(S))\subseteq \exc(S)$ for every $g\in S$, we get
\[
g_j(z) = B_j(z-a) + a \quad\text{whenever $\deg(g_j) = 1$},
\]
for some constant $B_j\in \C$ with $|B_j|>1$. From this discussion it follows that:
\begin{itemize}[leftmargin=14pt]
\item Every element of $S$ is of the form $B(z-a)^m +a$, where $B\in \C \setminus \{0\}$
and $m\in \Z_+$.

\item $S$ has degree-one elements, for all of which $a$ is a repelling fixed point.
\end{itemize}}

\section{Essential definitions and results in potential theory}\label{S:potential}
This section is devoted to presenting a number of essential definitions in potential theory
that we had deferred in Section~\ref{S:intro}. Additionally, we collect here several important 
results that we shall require for our proofs.

\begin{definition}\label{D:potential}
Let $\sigma$ be a Borel probability measure on $\C$ with compact support. Its 
\emph{logarithmic potential} is the function $U^{\sigma}: \C \to (-\infty,\infty]$ defined by 
\[
U^{\sigma}(z)=\int_{\C} \log \frac{1}{|z-t|} d\sigma(t)
\]
and its \emph{logarithmic energy} is given by 
\[
I(\sigma):= \int_{\C} \int_{\C} \log \frac{1}{|z-t|} d\sigma(z) d\sigma(t) 
=\int_{\C} U^\sigma(z)d\sigma(z).
\]
\end{definition}

The potential $U^{\sigma}$ is superharmonic in $\C$ and harmonic outside the support
of $\sigma$. As noted by Frostman \cite{frostman:potential35}, the potential $U^{\sigma}$ is finite at $z_0$
(i.e., does not take the value $+\infty$) if for some $\epsilon >0$ the integral 
\[
\int_0^\epsilon \frac{\sigma(D(z_0,r))}{r} dr
\]
exists and is finite. However, one can say much more.

\begin{result}[Arsove, \cite{arsove:cplmd60}]\label{R:arsove}
For the potential $U^{\sigma}$ of a Borel probability measure $\sigma$ to be 
continuous at $z_0$ it is necessary and sufficient that
\[
\lim_{\epsilon \to 0} \left \{\limsup_{z \to z_0} 
\int_0^\epsilon \frac{\sigma(D(z,r))}{r} dr \right\}=0.
\]
Moreover, if $\supp(\sigma)$ lies in a closed set $\Sigma$, then the approach of 
$z$ to $z_0$ in the above limit can be restricted to points $z$ of $\Sigma$.
\end{result}

\begin{remark}\label{rem:arsoveremark}
It follows from Result~\ref{R:arsove} that the potential $U^{\sigma}$ will be finite and 
continuous at $z_0$ if $\sigma$ satisfies a condition of the form
\[
\sigma(D(z,r)) \leq Cr^{\alpha} \quad \forall r\in (0, r_0),
\]
where $|z-z_0|<\delta$ and $C$, $\alpha$, $r_0$, $\delta$ are positive constants 
depending only on $\sigma$ and $z_0$.
\end{remark}

If $E\subset \C$ is a Borel set, then $\mathcal M (E)$ will always denote
the collection of all Borel probability measures $\sigma$ with $\supp(\sigma) \subset E$.
Let $K \subset \C$ be a compact subset of the complex plane. Define
\[
 V:= \inf \{I(\sigma):\sigma \in \mathcal M (K)\}.
\]
Then $V$ turns out to be finite or $+\infty$.
The quantity
\[
\text {cap}(K):= \exp(-V)
\]
is called the \emph{logarithmic capacity} (or simply \emph{capacity}) of $K$. 
The capacity of an arbitary Borel set $E$ is defined as 
\[
\cpt(E):=\sup \{ \cpt(K): K \subset E,\ \ K \text{ compact}\}
\]
and every set (not necessarily a Borel set) that is contained in a Borel set of zero capacity
is considered to have zero capacity. A property is said to hold 
\emph{quasi-everywhere} (which we shall often abbreviate to q.e.) on a set $E$ if the set of points
in $E$ at which this property does not hold is of logarithmic capacity zero.
\smallskip

The results and definitions that follow are from the book \cite{SaffTotik:lpwef97} by Saff and Totik.
First, we state a couple of results which describe the behaviour of the logarithmic potential 
with respect to a convergent sequence of measures in the weak* topology.

\begin{result}[Principle of Descent \& Lower Envelope Theorem]\label{R:PDLE}
Let $\sigma_n$, $n=1,2,\dots$, be a sequence of Borel probability measures all having support
in a fixed compact subset of $\C$. If $\sigma_n \to \sigma$ in the weak* topology then
\[
\begin{split}
\liminf_{n \to \infty} U^{\sigma_n}(z) &\geq U^{\sigma}(z) \qquad\text{ for every } z \in \C,\\
\liminf_{n \to \infty} U^{\sigma_n}(z) &= U^{\sigma}(z) \qquad\text{ for q.e. } z \in \C.
\end{split}
\]
\end{result}

We will now introduce some basic definitions and results from weighted potential theory.
\smallskip

Let $\Sigma \subset \C$ be a closed set and $w: \Sigma \to [0, \infty)$. 
We call such a function a \emph{weight function} on $\Sigma$.

\begin{definition}\label{D:weight}
A weight function $w$ on $\Sigma$ is said to be \emph{admissible} if it satisfies the following 
three conditions:
\begin{enumerate}[label=(\roman*), leftmargin=27pt]
\item $w$ is upper semi-continuous;
\item $\Sigma_0 := \{z \in \Sigma: w(z)>0\}$ has positive capacity;
\item if $\Sigma$ is unbounded, then $|z|w(z) \to 0$ as $|z| \to \infty, z \in \Sigma$.
\end{enumerate}
\end{definition}

Consider an admissible weight function on $\Sigma$, and define $Q\equiv Q_w$ by
\begin{equation}\label{E:w_Q_rel}
 w(z)=: \exp (-Q(z)).
\end{equation}
Then $Q: \Sigma \to (-\infty, \infty]$ is lower semi-continuous, $Q(z)< \infty$ on a set of 
positive capacity and if $\Sigma$ is unbounded, then 
\[
 \lim_{|z| \to \infty,z \in \Sigma} \{Q(z)-\log|z|\}=\infty.
\]
The function $Q$ is called an \emph{external field}.
\smallskip

Let $\Sigma\subset \C$ be a closed set.
For any $\sigma \in \mathcal M(\Sigma)$, and $w$ an admissible weight function on 
$\Sigma$, we define the \emph{weighted energy integral}
\[
\begin{split}
I_w(\sigma) & := \int_{\C} \int_{\C} \log\frac{1}{|z-t|w(z)w(t)} d\sigma(z)d\sigma(t)\\
& = \int_{\C} \int_{\C} \log \frac{1}{|z-t|} d\sigma(z)d\sigma(t) + 2 \int_{\C} Q d\sigma,
\end{split}
\]
where the last representation is valid whenever both integrals exist and are finite. 
It follows from the definition of an admissible weight that the first integral is well defined.

\begin{definition}\label{D:equi}
Let  $w$ be an admissible weight on the closed set $\Sigma$ and let
\[
 V_w:= \inf \{I_w(\sigma):\sigma \in \mathcal M (\Sigma)\}.
\]
Then a measure $\sigma$ is called an \emph{equilibrium measure associated with} $w$ 
(or, equivalently in view of \eqref{E:w_Q_rel}, an \emph{equilibrium measure associated with} $Q$) if 
\[
I_w(\sigma)= V_w.
\]
\end{definition}

\begin{remark}
In view of the relation \eqref{E:w_Q_rel}, we shall use the phrases ``admissible weight'' and
``external field'' interchangeably, both of which are standard in the literature.
We used the term ``external field'' in Section~\ref{S:intro} because it has a well-understood
meaning in electrostatics.
\end{remark}

Now we are ready to state the fundamental theorem of the theory, which gives existence and
uniqueness of equilibrium measures associated with $w$.

\begin{result}\label{R:fundamental}
Let $w$ be an admissible weight on the closed set $\Sigma$. Then $V_w$ is finite and there
exists a unique equilibrium measure $\sigma_w \in \mathcal M (\Sigma)$ associated with $w$. 
Moreover, $\sigma_w$ has finite logarithmic energy.
\end{result}

\begin{remark}\label{Rem:classical}
If $w\equiv 1$ (i.e., $Q\equiv 0$) then $I_w(\sigma)=I(\sigma)$.
In this case, if $\Sigma$ is a compact subset of $\C$ (of positive capacity) then we recover 
the classical theory of logarithmic potentials. The unique equilibrium measure associated with 
the weight $w\equiv 1$ on the compact set $\Sigma$ is called the \emph{equilibrium 
measure of $\Sigma$}. It is denoted by $\sigma_{\Sigma}$.
\end{remark}

With $w$ as above, define
\[
\begin{split}
\mathcal S_w&:= \supp(\sigma_w),\\
F_w&:= V_w - \int_{\C} Q d \sigma_w.
\end{split}
\]
The constant $F_w$ is called the \emph{modified Robin constant} for $w$.

\begin{result}\label{R:frost}
Let $w$ be an admissible weight on the closed set $\Sigma$. 
If $\sigma \in \mathcal M (\Sigma)$ has compact support 
and finite logarithmic energy, and
$U^{\sigma}(z)+Q(z)$
coincides with a constant $F$ for quasi-every $z$ in $\supp(\sigma)$, and 
$U^{\sigma}+Q \geq F$ quasi-everywhere on $\Sigma$, 
then $\sigma=\sigma_w$ and $F=F_w$.
\end{result}

Let $K$ be a compact subset of $\Sigma$ of positive capacity, and define
\begin{equation}\label{E:ms}
F(K):= \log \cpt(K)- \int_{\Sigma} Qd\sigma_{K},
\end{equation}
where $\sigma_{K}$ denotes the equilibrium measure of the compact set $K$.
This is the so-called \emph{$F$-functional of Mhaskar and Saff}, which is one of the most 
powerful tools in finding $\sigma_w$ and $\mathcal S_w$. We will use the $F$-functional to 
estimate the logarithmic capacity of the Julia set of a finitely generated polynomial semigroup.

\begin{result}\label{R:ms}
Let $w$ be an admisible weight on the closed set $\Sigma$. Then the following hold:
\begin{enumerate}[label=$(\alph*)$]
\item For every compact set $K \subset \Sigma$ of positive capacity, 
$F(K) \leq F(\mathcal S_w)$.
\item $F(\mathcal S_w)= -F_w$, where $F_w$ is the modified Robin constant for $w$.
\end{enumerate}
\end{result}

\section{Complex-analytic preliminaries}\label{S:complex}
This section gathers together a number of results in complex analysis, along with some
consequences thereof, that we shall need in our proofs in Sections~\ref{S:main}
and~\ref{S:bound}.

\subsection{Orders and degrees}
Let $f$ be a non-constant holomorphic $\sph$-valued map defined in a neighbourhood of 
$a \in \sph$. Let $(U, \phi)$ and $(V, \psi)$ be holomorphic charts at $a$ and $f(a)$
respectively such that $\tilde f := \psi \circ f \circ \phi ^{-1}$ is defined.
Suppose the Taylor expansion of $\tilde f$ at $\phi(a)$ has the form
\[
\tilde f(z)= b_0 +b_m(z-\phi(a))^m+b_{m+1}(z-\phi(a))^{m+1}+\dots,
\]
where $b_m\neq 0$. Recall that the (unique) integer $m$ does not depend on 
the choice of the charts $(U, \phi)$ or $(V, \psi)$, and
is called the \emph{order of $f$ at} 
$a$ and is denoted by $\ord_a(f)$.

\begin{result}[\textsc{Riemann--Hurwitz Formula}]\label{R:RH}
For any non-constant rational map $f$
\[
\sum_{z\in \sph}\big(\ord_z(f)-1\big) = 2 \deg(f) - 2.
\]
\end{result}

For a non-constant polynomial $g$, we have  $\ord_{\infty}(g)=\deg(g)$.
Thus, in this case, Result \ref{R:RH} becomes
\begin{equation}\label{E:RH}
\sum_{z\in \C} \big(\ord_z(g)-1\big) = \deg(g)-1.
\end{equation}
Observe that the general term in the sum is positive only when $z$ is a critical point of $g$.
\smallskip

Let the non-constant polynomials $g_1$, $g_2$, $\dots$, $g_N$ be the entries of the list
$A$ (whence they are \textbf{not} necessarily distinct). Recall the definition of $\critall(A)$: see \eqref{E:critall}.
In Section~\ref{S:main}, we will need to explicitly refer to these polynomials.
To this end, we define $\critall(g_1,g_2,\dots,g_N):= \critall(A)$.
\smallskip

The next lemma is needed in the proof of Proposition~\ref{P:card}. Recall, from the
discussion in Section~\ref{S:intro}, that Proposition~\ref{P:card} establishes for the semigroups 
of our interest a result analogous to that in \cite{BhaSri:hcrfgrs17}\,---\,but which allows elements of
$\gen$ to have critical points in $\J(S)$.
The following lemma is the key to dealing with the latter situation.

\begin{lemma}\label{L:choice}
Let $g_1$, $g_2$, $\dots$, $g_N$ be a collection of non-constant polynomials such that
$\deg(g_i) \geq 2$ for some $i\in \{1,\dots,N\}$.
Set $D:=\sum_{i=1}^N \deg(g_i)$.
If $\sharp\big(\critall(g_1,g_2,\dots,g_N)\big) > 1$ then there exists $\K \in \Z_+$ such that
\begin{equation}\label{E:choice}
\sum_{i:g_i'(x) \neq 0}{\left(\frac{D}{N}\right)}^{\frac{1}{\K}} 
+\sum_{i:g_i'(x) = 0} \ord_{x}(g_i) \leq D-\tfrac{1}{2} \quad \forall x\in \critall(g_1,g_2,\dots,g_N).
\end{equation}
\end{lemma}

\begin{proof}
For simplicity, we shall denote the set $\critall(g_1,g_2,\dots,g_N)$ as $\critall$.
Let $x \in \critall$. Suppose first that there exists $i^*\in \{1,\dots,N\}$ such that 
$\deg(g_{i^*}) \geq 2$ and $g_{i^*}'(x) \neq 0$. Then, since $\ord_{x}(g_{i^*})=1$, 
$\ord_{x}(g_{i^*}) \leq \deg(g_{i^*})-1$. On the other hand, if $g_i'(x)=0$ for every $g_i$ 
satisfying $\deg(g_i) \geq 2$ then, since $\sharp\critall > 1$, there exists 
$ x'\ (\neq x) \in \critall$ such that $g_{i'}'(x') = 0$ for some $i'\in \{1,\dots,N\}$. 
Thus $\ord_{x'}(g_{i'}) \geq 2$. Now, by \eqref{E:RH}, $\ord_{x}(g_{i'}) \leq \deg(g_{i'})-1$.
To summarize: for each $x \in \critall$ there exists 
$i\in \{1,\dots,N\}$ such that $\ord_{x}(g_i) \leq \deg(g_i)-1$. 
Of course, in general, we have 
$\ord_{x}(g_i) \leq \deg(g_i)$ for all $i\in \{1,\dots,N\}$. So summing over all 
$i\in \{1,\dots,N\}$ gives
\[
\sum_{i=1}^N\ord_x(g_i) \leq \left(\sum_{i=1}^N \deg(g_i) \right) -1=D-1.
\]
It is easy to see that for each $x \in \critall$,
\[
\sum_{i:g_i'(x) \neq 0}{\left(\frac{D}{N}\right)}^{\frac{1}{k}} 
+ \sum_{i:g_i'(x) = 0} \ord_{x}(g_i) \longrightarrow 
\sum_{i=1}^N\ord_x(g_i) \ \ \text{ as } \  k \rightarrow \infty.
\]
Thus, in view of the last inequality, we can choose $\K \in \Z_+$ such that
for all $x \in \critall$,
\[
\sum_{i:g_i'(x) \neq 0}{\left(\frac{D}{N}\right)}^{\frac{1}{\K}} 
+ \sum_{i:g_i'(x) = 0} \ord_{x}(g_i) \leq D-\tfrac{1}{2}.
\]
Thus we have the proof.
\end{proof}

\subsection{The Dinh--Sibony measure}\label{SS:DinSib}
We now provide a brief discussion of the formalism and the results underlying the
convergence statement \eqref{E:asymp1}. The first basic observation is that with $X$
as in Section~\ref{S:intro}, any two holomorphic correspondences on $X$ can be composed
with each other. Since compositions of correspondences of the most general kind are
not relevant to the proofs in this paper, we shall just make the following observations
on the subject of composing two correspondences. (For readers who are more comfortable
with complex analysis in one dimension, we refer to \cite{BhaSri:hcrfgrs17} for a more
detailed discussion.) They are:
\renewcommand{\theenumi}{{\emph{\roman{enumi}}}}
\begin{enumerate}[leftmargin=27pt]
\item The \emph{topological degree} of a holomorphic correspondence $\Gamma$ is the
generic number of preimages of a point counted according to multiplicity. To elaborate:
representing $\Gamma$ as in
\eqref{E:stdForm}, it is classical that there is a Zariski-open
set $W\subset X_2$ and $\nu_i\in \Z_+$ such that $(\pi_2^{-1}(W)\cap \Gamma_i, W, \pi_2)$
is a $\nu_i$-sheeted covering.
The topological degree of $\Gamma$ is defined as
\[
{\rm deg}_{top}(\Gamma) := \sum\nolimits_{1\leq i\leq N} m_i\nu_i.
\]
In the $1$-dimensional case, we abbreviate ${\rm deg}_{top}(\Gamma)$ to $d_1(\Gamma)$,
as introduced in Section~\ref{S:intro}. In what follows, $\Gamma^{\dagger}$ will denote the
\emph{adjoint} of $\Gamma$. In the notation of  \eqref{E:stdForm},
\[
\Gamma^{\dagger} := \sum\nolimits_{1\leq i\leq N}\!m_i\Gamma^{\dagger}_i
\]
where $\Gamma^{\dagger}_i := \{(y,x)\in X_2\times X_1: (x,y)\in \Gamma_i\}$.
In the $1$-dimensional case, $d_0(\Gamma) := d_1(\Gamma^{\dagger})$.

\item A holomorphic correspondence $\Gamma$ determines a relation from
$X_1$ to $X_2$ given, in the notation of \eqref{E:stdForm}, by $\cup_{1\leq i\leq N}\Gamma_j$.
Thus, given two holomorphic correspondences on $X$, their composition is, in essence, the
composition of the underlying relations with a little care taken to account for the multiplicities
(the integers $m_1, m_2,\dots, m_N$ in the notation of \eqref{E:stdForm}). The observation
concerning this ``accounting'' that is relevant to us is that if $X$ is a compact Riemann surface
and  $\Gamma^1$ and $\Gamma^2$ are holomorphic correspondences on $X$ then
$d_j(\Gamma^2\circ \Gamma^1) = d_j(\Gamma^2)d_j(\Gamma^1)$, $j = 0, 1$.
We shall denote the $n$-fold iterated composition of $\Gamma$ by $\Gamma^{\circ{n}}$.
\end{enumerate}
Also of immediate relevance is the following formula: given the following collections of
non-constant rational maps $g_1, g_2,\dots, g_N$ and $f_1, f_2,\dots, f_M$, not necessarily
distinct, and
\[
\Gamma^1 := \sum_{1\leq i\leq N}{\rm graph}(g_i) \quad\text{and}
\quad \Gamma^2 := \sum_{1\leq j\leq M}{\rm graph}(f_j)
\]
it turns out that
\[
\Gamma^2\circ \Gamma^1 = \sum_{1\leq i\leq N}\,\sum_{1\leq j\leq M}{\rm graph}(f_j\circ g_i)
\] 
Deferring for the moment the discussion of pullbacks of currents by holomorphic
correspondences, we fix the following notation: with $X$ as in Section~\ref{S:intro},
$\Gamma$ a holomorphic correspondence on $X$ and $T$ a current that can be
pulled back by $\Gamma$, we will denote the pullback of $T$ by $F_{\Gamma}^{*}T$.
With this, we can state the two results from which \eqref{E:asymp1} follows. These are
results by Dinh--Sibony. The specific results cited actually address much more general 
(including multi-dimensional) situations than ours. In the form in which they appear, they
are heavily paraphrased in two ways:
\begin{itemize}[leftmargin=14pt]
\item they are stated merely for holomorphic correspondences on $\sph$; and
\item the convergence stated below actually holds on a larger class of test functions
(which were introduced in \cite{DinhSibony:dvtma06}; also see \cite{DinhSibony:emtddc05}),
but weak${}^{\boldsymbol{*}}$ convergence suffices for our purposes.
\end{itemize} 
With these words, the results needed are: 

\begin{result}[Th{\'e}or{\`e}me~5.1 of \cite{DinhSibony:dvtma06} paraphrased for
$\sph$]\label{R:1st_DS_Res}
Let $\Gamma_n$, $n\in \Z_+$, be holomorphic correspondences on $\sph$. Suppose that the series
\[
\sum\nolimits_{n\in \Z_+}(d_0(\Gamma_1)/d_1(\Gamma_1))\dots(d_0(\Gamma_n)/d_1(\Gamma_n))
\]
converges. Then, there exists a regular Borel probability measure $\mu$ such that
\[
 d_1(\Gamma_1)^{-1}\!\!\dots d_1(\Gamma_n)^{-1}
 F^*_{\raisebox{-1pt}{\!$\scriptstyle \Gamma_n\circ\dots\circ\Gamma_1$}}(\FS)
 \weakST \mu \;\; \text{as measures, as $n\to \infty$.}
\]
The measure $\mu$ places no mass on polar sets.
\end{result}

\begin{result}[Th{\'e}or{\`e}me~1.1 of \cite{DinhSibony:dvtma06} paraphrased for
$\sph$]\label{R:2nd_DS_Res}
Let $\Gamma_n$, $n\in \Z_+$, be holomorphic correspondences on $\sph$. Suppose
$\sum_{n\in \Z_+}d_0(\Gamma_n)/d_1(\Gamma_n)$ converges. Then, there exists a 
Borel polar set $E\varsubsetneq \sph$
such that for any $a\in \sph\setminus E$,
\[
 d_1(\Gamma_n)^{-1}\!\big(F^*_{\raisebox{-1pt}{\!$\scriptstyle \Gamma_n$}}(\FS) -
 F^*_{\raisebox{-1pt}{\!$\scriptstyle \Gamma_n$}}(\delta_a)\big)
 \weakST 0 \;\; \text{as $n\to \infty$.}
\]
\end{result}

In both these results, $\FS$ stands for the Fubini--Study form on $\sph$. As this is a volume
form on $\sph$, it and its pullbacks are treated as measures in Result~\ref{R:1st_DS_Res}.
Both results involve the  notion of the pullback of a measure by a correspondence.
A measure\,---\,as discussed in \cite[Sections~2--3]{DinhSibony:dvtma06}\,---\,is
an example of a current that can be pulled back by a holomorphic correspondence,
which is the general framework for the results in \cite{DinhSibony:dvtma06}. Since there
is a fairly detailed discussion of the definition and computation of pullbacks of measures
in the one-dimensional setting in \cite[Section~4.1]{BhaSri:hcrfgrs17}, we refer the
reader to it.
\smallskip

To conclude this section, we present the following pullback formula (the
details of whose computation are presented in the last reference). Let
$g_1, g_2,\dots, g_N$ be a collection of non-constant polynomials, not necessarily
distinct. Call this collection $\mathscr{C}$ and write
\[
\corrC := \sum\nolimits_{1\leq i\leq N}{\rm graph}(g_i).
\]
Let $a\in \sph$. For simplicity of notation, we shall abbreviate here, and in the 
sections that follow, the pullback
$F_{\Gamma_{\!\!\!\genfrac{}{}{0pt}{2}{}{\mathscr{C}}}}^{*}\!\!\delta_a$ as
$F_{\mathscr{C}}^{*}\delta_a$. Then,
\begin{equation}\label{E:dirac_example}
F_{\mathscr{C}}^{*}\delta_a = \sum_{1\leq i\leq N}\,\sum_{x \in {g_i^{-1}\{a\}}^{\bullet}}\delta_x.
\end{equation}
Here, the notation $x \in {g_i^{-1}\{a\}}^{\bullet}$ signifies that $x$ is repeated according to multiplicity
as it varies through $g_i^{-1}\{a\}$. Also, let us abbreviate
$F_{\Gamma_{\!\!\!\genfrac{}{}{0pt}{2}{}{\mathscr{C}}}^{\circ n}}^{*}\delta_a$ as
$(F_{\mathscr{C}}^{*})^n\delta_{a}$. Then, from Results~\ref{R:1st_DS_Res} and~\ref{R:2nd_DS_Res},
we conclude that there exist a Borel polar set $E\varsubsetneq \sph$ and a measure $\mu_{\mathscr{C}}$
having the properties stated in Result~\ref{R:1st_DS_Res} such that
\begin{equation}\label{E:asymp2}
d_1(\corrC)^{-n}(F_{\mathscr{C}}^{*})^n\delta_{a}\weakST \mu_{\mathscr{C}}
 \;\; \text{as $n\to \infty$,} \;\; \forall a\in \sph\setminus E.
\end{equation}

\section{Theorem~\ref{T:main} and associated results}\label{S:main}
This section is devoted to proving several results\,---\,including the theorem alluded to in
Section~\ref{S:intro} in the discussion following the statement of Theorem~\ref{T:main}\,---\,that
are closely tied to the proof of the latter theorem. To do so, we need to fix certain notations.
Let $S$ be a finitely generated polynomial semigroup. Let
$g_1$, $g_2$, $\dots$, $g_N$ (\textbf{not} necessarily distinct) be polynomials
such that $S = \langle g_1, g_2,\dots, g_N\rangle$, and
define
\[
M:=\max \{ |g_i'(z)|: z\in \J(S), \ i \in \{1, \dots, N\}\},
\]
\[
 R:=\frac{D}{N} \ \ \ \text{and} \ \ \ \lambda:= \frac {\log R}{\log M},
\]
where $D:=\sum_{i=1}^N \deg(g_i)$.
Thus $M=R^{\frac {1}{\lambda}}$. Note that $R>1$ and by Result~\ref{R:HM}, 
repelling fixed points of all elements of $S$ are dense in $\J(S)$. Thus $M>1$. Also define
\[
\crit(g_1,g_2,\dots,g_N) := \{c \in \J(S):  g_i'(c)=0 \text{ for some } i\in \{1,\dots,N\}\}.
\]

In what follows, a collection denoted by $A^{\bullet}$ will represent a list; the objects in $A^{\bullet}$ 
will be repeated according to multiplicity.
Also $A$ will denote the set underlying $A^{\bullet}$. The notation $\sharp A^{\bullet}$ 
will denote the number of objects in $A^{\bullet}$, counted according to multiplicity. All other notations
will be as introduced in Sections~\ref{S:intro} and~\ref{S:complex}.
\smallskip

We are now ready to state the following

\begin{proposition}\label{P:card}
Let $S$ be a finitely generated polynomial semigroup and let
$g_1$, $g_2$, $\dots$, $g_N$ (not necessarily distinct) be polynomials
such that $S = \langle g_1, g_2,\dots, g_N\rangle$.
Assume
\[
\sharp\big(\critall(g_1, g_2,\dots, g_N)\big) > 1.
\]
Consider the correspondence
\[
\Gamma :=\sum\nolimits_{1\leq i\leq N} {\rm graph}(g_i),
\]
and abbreviate ${(F^{n})}^{\dagger} := F_{(\Gamma^{\dagger})^{\circ n}}$. Then
there exist $r_0>0$ and $\K \in \Z_+$ such that for any $ r \in (0,r_0]$ and 
$\arb \in \J(S)$, we have 
\begin{equation}\label{E:card}
\sharp({(F^{n})}^{\dagger}(\arb) \cap D(z,r))^{\bullet} 
\leq \max \big( D^{n-\frac{\sub}{\K}+1}N^{\frac{\sub}{\K}-1}, {(D-\tfrac{1}{2})}^{n}\big)
\end{equation}
for all $n \in \mathbb N$ and $z \in \C$, where $\sub \in \Z_+$ is the unique integer such that 
\[
 r \in I(\sub):=\left (r_0 R^{\frac{-2\sub}{\lambda}}, r_0 R^{\frac{-2(\sub-1)}{\lambda}} \right].
\]
\end{proposition}

\begin{proof}
For $\eps > 0$, let us write
\[
 \J^{\eps}:= \bigcup\nolimits_{\xi \in \J(S)}D(\xi, \eps) \ \ \ \text{and} \ \ \ 
\bar \J^{\eps}:= \overline{\J^{\eps}}.
\]
In this proof we will abbreviate $\crit(g_1,g_2,\dots,g_N)$ to $\crit$. If
$\crit\neq \emptyset$ then denote this (finite) set of points by $\{c_1, c_2, \dots, c_q\}$.
Note that \eqref{E:choice} holds for all $c \in \crit$. Consider the following
quantities:
\begin{itemize}[leftmargin=14pt]
\item If $\crit\neq \emptyset$ then let $\delta_1 > 0$ be so small that:
\renewcommand{\labelenumi}{{\arabic{enumi})}}
  \begin{enumerate}
  \item $D(c_j,2\delta_1)$ are pairwise disjoint for $j=1,2,\dots,q$,
  \item if for some $i\in \{1,\dots,N\}$ and $j\in \{1,\dots,q\}$, $g_i'(c_j) =0$ then $|g_i'(z)| \leq
  1$ for every 
  $z \in D(c_j,2\delta_1) $ and $g_i$ maps at most $\ord_{c_j}(g_i)$ points of 
  $D(c_j,2\delta_1)$ to a single point of $\C$,
  \item if for some $i\in \{1,\dots,N\}$ and $j\in \{1,\dots,q\}$, $g_i'(c_j) \neq 0$ then
  $|g_i'(z)| \neq 0$ for 
  every $z \in D(c_j,2\delta_1)$.
  \end{enumerate}
  If $\crit = \emptyset$, then we just set $\delta_1:=1$.
  \smallskip
  
\item Let $\delta_2 > 0$ be such that $g_i'(z) \neq 0$ for every 
  $z \in \J^{2 \delta_2} \setminus \J(S)$ and $i=1, \dots, N$.
  \smallskip
  
\item Let $\delta_3 > 0$ be such that $|g_i'(z)| < R^{\frac {2}{\lambda}}$ for every 
  $z \in \J^{\delta_3}$ and $i=1, \dots, N$.
\end{itemize}
Next, we introduce the following open covers:
\renewcommand{\theenumi}{{\emph{\roman{enumi}}}}
  \begin{enumerate}[leftmargin=27pt]
 \item\label{I:big_cover} Define
 \[
 \mathcal O_{0}:= \begin{cases}
 					\{ D(\xi, r(\xi)): \xi \in \bar \J^{\delta_2}
 					\setminus \cup_{j=1}^q D(c_j, \delta_1) \}, &\text{if $\crit\neq \emptyset$}, \\
 					\{ D(\xi, r(\xi)): \xi \in \bar \J^{\delta_2} \}, &\text{if $\crit= \emptyset$},
				\end{cases}
  \]
  where $r(\xi)>0$ is such that $g_i|_{ D(\xi, r(\xi))}$ is injective for $i=1,2,\dots,N$.
  \smallskip
  
  \item\label{I:ball_cover} If $\crit\neq \emptyset$ then for each $j=1,2,\dots,q$, 
   $\mathcal O_j := \{ D(\xi, r(\xi)): \xi \in \overline { D(c_j, \delta_1)} \}$,
   where $r(\xi)>0$ is such that if $g_i'(c_j) \neq 0$ then $g_i|_{ D(\xi, r(\xi))}$ is injective.
  \end{enumerate}
Finally, we set $\delta_4 > 0$ to be the minimum of the Lebesgue numbers of 
all the covers introduced above, and write:
\[
r_0:= \frac {\min \{\delta_1, \delta_2, \delta_3,\delta_4 \}}{4}.
\]

If $\crit = \emptyset$ then set $\K = 1$, else
fix $\K \in \Z_+$ such that \eqref{E:choice} holds true for $g_1, g_2,\dots, g_N$ as
above. With this choice of $r_0$ and $\K$, we will prove our result by induction on $n$.
To this end, we must mention that, by definition, ${(F^{0})}^{\dagger}(\arb) := \{\arb\}$
for any $\arb\in \sph$.
\smallskip

Fix an arbitary $\arb \in \J(S)$.
If $n=0$ then the second term on the right-hand side of \eqref{E:card} is 1 and the left-hand 
side of \eqref{E:card} is by definition $\leq$ 1 for any $z\in \C$ and $r\in (0, r_0]$.
Hence, the inequality \eqref{E:card} follows for these $r$ and $z$.
Now assume that \eqref{E:card}
holds for $n=m$, where $m \geq 0$, for every $z\in \C$ and $r\in (0, r_0]$.
We will study what this implies for $n=m+1$. Consider an arbitrary
$z\in \C$ and $r\in (0, r_0]$. 
If $(z,r)$ is such that $D(z,r) \cap \J(S) = \emptyset$ then as
${(F^{n})}^{\dagger}(\arb) \subset \J(S)$,
the left-hand side of \eqref{E:card} is zero and the latter inequality trivially follows
for all $n\in \N$. Hence, 
assume that $D(z,r) \cap \J(S) \neq \emptyset.$
\smallskip

If, for the $r$ chosen, $\sub=1$ then the first term on the right-hand side
of \eqref{E:card} is 
\[
 D^{m+1}{\Big(\frac{D}{N}\Big)}^{1-\frac{1}{\K}} \geq  D^{m+1},
\] 
whence, by definition, the left-hand side of \eqref{E:card} $\leq D^{m+1}$. Thus 
\eqref{E:card} (for $\nu=1$) follows. We therefore consider $\sub > 1$.
\smallskip

\noindent \textbf{Case 1.} \emph{Either $\crit=\emptyset$ OR
$\big(\crit\neq \emptyset$ and $z \not \in \cup_{j=1}^q D(c_j,\delta_1)\big)$.}
\vspace{1mm}

\noindent{Observe that $D(z,r) \cap \J(S) \neq \emptyset$ implies that $z \in \J^r$.
Let us write
\[
\bad := \begin{cases}
		\cup_{j=1}^q D(c_j,\delta_1), &\text{if $\crit\neq \emptyset$}, \\
		\emptyset, &\text{if $\crit = \emptyset$}.
	     \end{cases}
\]
Since $r \leq \delta_2/4$, we get that $z \in \bar \J^{\delta_2} \setminus \bad$.
Also, since $r \leq \delta_4/4$, there exists a disc 
$ D(\xi, r(\xi))\in \mathcal O_{0}$ such that $D(\xi, r(\xi))\supset D(z,r)$.
Thus, by the choice of $r(\xi)$ in defining the cover $\mathcal O_{0}$, $g_i|_{ D(z, r)}$ 
is an injective map for $i=1,\dots,N$, and we get
\[
\sharp({(F^{m+1})}^{\dagger}(\arb) \cap D(z,r))^{\bullet} 
= \sum_{i=1}^N  \sharp({(F^{m})}^{\dagger}(\arb) \cap g_i(D(z,r)))^{\bullet}.
\]
Now, $r \leq \delta_3/4$ and $z \in \J^r$ implies that $D(z,r) \subset \J^{\delta_3}$. Thus
$|g_i'(\zt)| < R^{{2}/{\lambda}}$ for all 
$\zt \in D(z,r)$ and $i=1,\dots, N$. Therefore, by an application of
the Mean Value Inequality, we get 
$g_i(D(z,r)) \subset D(g_i(z), rR^{{2}/{\lambda}})$.
Thus
\[
\sum_{i=1}^N  \sharp({(F^{m})}^{\dagger}(\arb) \cap g_i(D(z,r)))^{\bullet}
\leq \sum_{i=1}^N \sharp({(F^{m})}^{\dagger}(\arb) 
\cap D(g_i(z),rR^{\frac{2}{\lambda}}))^{\bullet}.
\]
Observe that $rR^{{2}/{\lambda}} \in I(\sub-1)$; 
thus, by the induction hypothesis and the above observations,
\begin{align}
 \sharp({(F^{m+1})}^{\dagger}(\arb) \cap D(z,r))^{\bullet} 
 & \leq N \max \big( D^{m-\frac{\sub-1}{\K}+1}N^{\frac{\sub-1}{\K}-1}, 
     {(D-\tfrac{1}{2})}^{m}\big) \notag \\
 & = \max \Big( D^{m+1-\frac{\sub}{\K}+1}N^{\frac{\sub}{\K}-1}
     \Big(\frac{N}{D}\Big)^{1-\frac{1}{\K}}, N{(D-\tfrac{1}{2})}^{m}\Big). \label{E:case-1_ineq}
\end{align}
As $N < D$, we see that
\[
N \leq D-\tfrac{1}{2} \ \  \text{ and } \ \  \Big( \frac{N}{D}\Big)^{1-\frac{1}{\K}} \leq 1.
\]
Thus, from \eqref{E:case-1_ineq} and the above, we have the desired claim for $n=m+1$:
\[
\sharp({(F^{m+1})}^{\dagger}(\arb) \cap D(z,r))^{\bullet} 
\leq \max \big( D^{m+1-\frac{\sub}{\K}+1}N^{\frac{\sub}{\K}-1}, {(D-\tfrac{1}{2})}^{m+1}\big).
\]}

When $\crit\neq \emptyset$, we must also consider
\smallskip
 
\noindent \textbf{Case 2.} $z \in \cup_{j=1}^q D(c_j,\delta_1)$.
\vspace{1mm}

\noindent{By our choice of $\delta_1 > 0$, there is a unique $j^0\in \{1, \dots, q\}$ such that
$z \in D(c_{j^0},\delta_1)$. For simplicity, we shall write $c := c_ {j^0}$.
If $i\in \{1, \dots, N\}$ is such that $g_i'(c) \neq 0$ then
the description in (\ref{I:ball_cover}) of the cover for $\overline { D(c, \delta_1)}$
implies that  $g_i|_{D(z,r)}$ is an injective map. Note that $D(z,r) \subset D(c, 2{\delta}_1)$. 
Thus, from our discussion on the choice of $\delta_1>0$, if $i\in \{1, \dots, N\}$ is such that $g_i'(c) = 0$ then 
$g_i$ maps at most $\ord_{c}(g_i)$ points of $D(z,r)$ to 
a single point of $\C$. Thus we have 
\[
\begin{split}
 \sharp({(F^{m+1})}^{\dagger}(\arb) \cap D(z,r))^{\bullet} 
 & \leq \sum\nolimits_{1\leq i \leq N}'  \sharp({(F^{m})}^{\dagger}(\arb) \cap  g_i(D(z,r)))^{\bullet} \\
 & \qquad + \sum\nolimits_{1\leq i \leq N}^{*} \ord_c(g_i)\,\sharp({(F^{m})}^{\dagger}(\arb) \cap g_i(D(z,r)))^{\bullet}.
\end{split}
\]
In the above inequality, the primed sum denotes the sum over \textbf{only} those indices $i$ such that
$g_i'(c) \neq 0$ while the starred sum denotes the sum over \textbf{only} those indices $i$ such that $g_i'(c) = 0$.
These will have the same meaning in the expressions below.}
\smallskip

If $g_i'(c) \neq 0$ then $|g_i'(\zt)| \leq R^{{2}/{\lambda}}$ for all $\zt \in D(z,r)\subset \J^{\delta_3}$.
Now, by the Mean Value Inequality, $g_i(D(z,r))\subset D(g_i(z), rR^{{2}/{\lambda}})$.
Similarly, if $g_i'(c) = 0$ then $|g_i'(\zt)| \leq 1$ for all $\zt \in  D(z,r)\subset D(c,2\delta_1)$, 
and we get 
$g_i(D(z,r)) \subset D(g_i(z),r)$.
Thus:
\[
\begin{split}
 \sharp({(F^{m+1})}^{\dagger}(\arb) \cap D(z,r))^{\bullet} 
 &\leq\sum\nolimits_{1\leq i \leq N}'\sharp({(F^{m})}^{\dagger}(\arb)\cap D(g_i(z),rR^{\frac{2}{\lambda}}))^{\bullet}\\
 &\qquad + \sum\nolimits_{1\leq i \leq N}^{*} \ord_c(g_i)\,\sharp({(F^{m})}^{\dagger}(\arb) \cap  D(g_i(z),r))^{\bullet}.
\end{split}
\]
Now, by the induction hypothesis and noting that $rR^{{2}/{\lambda}} \in I(\sub-1)$,
\[
\begin{split}
  \sharp({(F^{m+1})}^{\dagger}(\arb) \cap D(z,r))^{\bullet}
 &\leq \sum\nolimits_{1\leq i \leq N}' \max \big(D^{m-\frac{\sub-1}{\K}+1}N^{\frac{\sub-1}{\K}-1}, 
    {(D-\tfrac{1}{2})}^{m}\big)\\
 &\qquad+\sum\nolimits_{1\leq i \leq N}^{*}\ord_c(g_i)\,\max\big( D^{m-\frac{\sub}{\K}+1}N^{\frac{\sub}{\K}-1}, 
     {(D-\tfrac{1}{2})}^{m} \big)\\
 &\leq \left(\sum\nolimits_{1\leq i \leq N}' \Big( \frac{D}{N}\Big)^{\frac{1}{\K}}+
    \sum\nolimits_{1\leq i \leq N}^{*}\ord_c(g_i)\right) \\
 & \qquad \qquad \times\max \big( D^{m-\frac{\sub}{\K}+1}N^{\frac{\sub}{\K}-1}, 
     {(D-\tfrac{1}{2})}^{m}\big).
\end{split}
\]
Now, by our choice of $\K$ and Lemma~\ref{L:choice}, we get
\[
\begin{split}
\sharp({(F^{m+1})}^{\dagger}(\arb) \cap D(z,r))^{\bullet}
 &\leq  (D-\tfrac{1}{2}) \max \big( D^{m-\frac{\sub}{\K}+1}N^{\frac{\sub}{\K}-1}, 
    {(D-\tfrac{1}{2})}^{m} \big)\\
 &\leq \max \big( D^{m+1-\frac{\sub}{\K}+1}N^{\frac{\sub}{\K}-1}, {(D-\tfrac{1}{2})}^{m+1} \big).
\end{split}
\]
Thus we have the desired claim for $n = m+1$ in this case too.
\smallskip

From Cases 1 and 2, \eqref{E:card} is true for $n=m+1$. By induction, 
\eqref{E:card} is true for all $n \in \mathbb N$. Since $\arb \in \J(S)$ was arbitary, 
the proof is complete.
\end{proof}

\begin{remark}\label{rem:crit_empty}
We saw in the above proof that in case, for
$g_1,g_2,\dots,g_N$ as in Proposition~\ref{P:card},
$\crit(g_1,g_2,\dots,g_N) = \emptyset$, then the situation discussed in Case~2 does not even arise.
Observe that, in this circumstance, the condition $\sharp\big(\critall(g_1,g_2,\dots,g_N)\big) > 1$ is
irrelevant. In short: if
$\crit(g_1,g_2,\dots,g_N) = \emptyset$ then the conclusion of Proposition~\ref{P:card} holds true with
no conditions on $\critall(g_1,g_2,\dots,g_N)$. This follows from the argument presented under
Case~1. That, in essence, is the argument in \cite{BhaSri:hcrfgrs17}.
\end{remark} 

We now present a result that, apart from being central to the proof of Theorem~\ref{T:main}, might
be of independent interest.

\begin{theorem}\label{T:conti}
Let $S$ be a finitely generated polynomial semigroup. Suppose $S$ satisfies the property that
if $\sharp\critall[S] = 1$ then $\crit[S]\cap \exc(S)=\emptyset$. Then, for any finite set of generators 
$\gen$ of $S$, the potential $U^{\mu_{\gen}}$ is finite and continuous on $\C$.
\end{theorem}
\begin{proof}
We fix a set of generators $\gen=\{ g_1, g_2,\dots,g_N\}$ of $S$.
Let $E(\gen)$ denote the Borel polar set associated with $\gen$ described just prior to
\eqref{E:asymp2}. It is a classical fact that for any 
polynomial $g$ with $\deg(g)\geq 2$,
\[
\cpt(\J(g))=|\lead{g}|^{\frac{1}{1-\deg(g)}}>0.
\]
Since for any $g\in S$ we have $\J(g) \subset \J(S)$, it follows that
$\cpt(\J(S))>0$. Therefore, $\J(S)\setminus E(\gen)\neq \emptyset$. Hence, we
can pick a $\fixed\in \J(S)\setminus E(\gen)$, which we shall \textbf{fix}
for the remainder of this proof.
Write $\mu_n := \mu^{\fixed}_n:= D^{-n}(F_{\gen}^*)^n(\delta_{\fixed})$
(recall the notation introduced in Section~\ref{SS:DinSib}).
Then, by \eqref{E:asymp2},
\begin{equation}\label{E:asymp3}
\mu_n \weakST \mu_{\gen} \;\; \text{as $n\to \infty$.}
\end{equation}

To begin with, we consider
the case when $\sharp\big(\critall(\gen)\big) > 1$. We apply Proposition~\ref{P:card}
to $g_1, g_2,\dots, g_N$ that we have fixed above.
Then, by this proposition, there exists $r_0>0$
such that if $r \in (0,r_0]$ then
\begin{equation}\label{E:mu_n_bound}
\sharp({(F^{n})}^{\dagger}(\fixed) \cap D(z,r))^{\bullet} 
\leq \max \big( D^{n-\frac{\sub}{\K}+1}N^{\frac{\sub}{\K}-1}, {(D-\tfrac{1}{2})}^{n}\big)
\end{equation}
(using the abbreviated notation in Proposition~\ref{P:card}) for any $z\in \C$.
Then, in view of the formula
\eqref{E:dirac_example} and \eqref{E:mu_n_bound}, we have for
$n$ sufficiently large:
\[
\mu_n (D(z,r))  =\frac{1}{D^n} \sharp({(F^{n})}^{\dagger}(\fixed) \cap D(z,r))^{\bullet} 
\leq \Big(\frac {D}{N} \Big)^{1-\frac{\sub}{\K}},
\]
where $\sub$ is the unique integer such that $r \in I(\sub)$\,---\,the latter as introduced in
the statement of Proposition~\ref{P:card}.  Thus $r>r_0 R^{{-2\sub}/{\lambda}}$. Recalling
that $R := D/N$, the last two inequalities give 
\begin{equation}\label{E:meas_bound}
\mu_n (D(z,r))\leq \bigg(\frac{R}{{r_0}^{\lambda/2\K}}\bigg)\,r^{\frac {\lambda}{2\K}}=C_1r^\alpha
\end{equation}
for all $n$ sufficiently large, where $C_1:={R}/({{r_0}^{{\lambda}/{2\K}}})>0$, $\alpha:={\lambda}/{2\K}>0$, 
$r \in (0,r_0]$ and $z \in \C$.
\smallskip

From \eqref{E:asymp3} and \eqref{E:meas_bound}, we see that for every $r \in (0,r_0]$,
\[
\mu_{\gen}(D(z,r)) \leq C_1r^\alpha.
\]
Invoking Remark~\ref{rem:arsoveremark}, $U^{\mu_{\gen}}$ is finite and continuous on $\C$.
\smallskip

Now consider the case when $\sharp\big(\critall(\gen)\big) = 1$. In this case,
by Lemma~\ref{L:crit_condtn}, $\crit(\gen)\cap \exc(S)=\emptyset$. 
If $\crit(\gen) = \emptyset$ then, by Remark~\ref{rem:crit_empty}, the conclusion of
Proposition~\ref{P:card} still holds true. Thus we get \eqref{E:meas_bound} (with $\K = 1$ this time).
Consequently, arguing as before, $U^{\mu_{\gen}}$ is finite and continuous.
\smallskip

 It remains to consider the case when
$\sharp\big(\critall(\gen)\big) = 1$, $\crit(\gen) \cap \exc(S)=\emptyset$ and 
$\crit(\gen) \neq \emptyset$. Since $\sharp\big(\critall(\gen)\big) = 1$
and $\crit(\gen) \neq \emptyset$, we get that
\begin{equation}\label{E:crit_eq} 
\crit(\gen)=\critall(\gen),
\end{equation}
Consider the holomorphic correspondence associated with the \textbf{list} of polynomials
$\gen^2:=\{g_i \circ g_j: 1\leq i,j \leq N\}^{\bullet}$,  i.e.,
\[
\Gamma_{\gen^2}:=\sum_{1\leq i,j \leq N} \text{graph}(g_i\circ g_j).
\]
Also, let $S' := \langle g_i \circ g_j: 1\leq i,j \leq N \rangle$.
\smallskip

It is easy to see that for $g\circ g\in S'$ for any $g\in S$. From this and the classical fact
that $\J(g) = \J(g\circ g)$, we deduce\,---\,in view of Result~\ref{R:HM}\,---\,that
$\J(S') = \J(S)$. Thus the $\fixed$ that we had fixed above belongs to $\J(S')$. 
Write $\mu_n' := D^{-2n}\big(F^*_{\!\gen^2}\big)^n(\delta_{\fixed})$.
It is easy to see that $\mu_n'=\mu_{2n}$ for all 
$n \in \Z_+$. Thus, by \eqref{E:asymp3}:
\begin{equation}\label{E:asymp4}
\mu_n' \weakST \mu_{\gen} \;\; \text{as $n\to \infty$.}
\end{equation}
Now, let
\[
\critall(\gen^2):=\{c \in \C:  (g_i \circ g_j)'(c)=0 \text{ for some } i,j\in \{1,\dots,N\}\}.
\]
Observe that for the list $\gen^2$,
\begin{equation}\label{E:critall_2nd_iterate}
\critall(\gen^2)=\critall(\gen)\cup \Big[\cup_{i=1}^N g_i^{-1}\big(\critall(\gen)\big)\Big] 
=\crit(\gen)\cup \Big[\cup_{i=1}^N g_i^{-1}\big(\crit(\gen)\big)\Big] ,
\end{equation}
where the second equality is a consequence of \eqref{E:crit_eq}.
We now argue that $\sharp\big(\critall(\gen^2)\big) > 1$.
Since $\crit(\gen)\cap \exc(S)=\emptyset$ and $\sharp\big(\crit(\gen)\big)=1$, there exists $x$
such that $x \notin \crit(\gen)$ but $g_i(x) \in \crit(\gen)$ for some $i\in \{1,\dots,N\}$, i.e.,
\[
x\in \cup_{i=1}^N\,g_i^{-1}\big(\crit(\gen)\big).
\]
Consequently, by \eqref{E:critall_2nd_iterate},
$\sharp\big(\critall(\gen^2)\big) >1$. Thus Proposition~\ref{P:card} can be applied
to the correspondence $\Gamma_{\gen^2}$ and, as the $\fixed$ we had fixed
lies in $\J(S')$, its conclusion applies to this $\fixed$.
By a computation analogous to the one in the second paragraph of this proof\,---\,with 
$\mu_n'$ replacing $\mu_n$ and using \eqref{E:asymp4}\,---\,we deduce that
there exists $r_0>0$ such that 
if $ r \in (0,r_0]$ then
\[
\mu_{\gen}(D(z,r)) \leq C_2r^{\beta}
\]
for some positive constants $C_2$ and $\beta$.
Once again, by Remark~\ref{rem:arsoveremark}, $U^{\mu_{\gen}}$ is finite and 
continuous on $\C$
in this final case as well.
\end{proof}

Recall that for a polynomial $g$, $\lead{g}$ denotes the coefficient of the highest degree term 
of the polynomial $g$.

\begin{lemma}\label{L:compu}
Let $S$ be a finitely generated polynomial semigroup and let $\gen=\{ g_1, g_2,\dots,g_N\}$ 
be a set of generators of $S$ then
\[
\lim_{n \to \infty} \frac{1}{D^n} \log \Bigg(\prod_{l(g)=n}\!|\lead{g}|\Bigg) \ \text{ exists }
\]
and equals $(D-N)^{-1}\log |\lead{g_1}\lead{g_2}\dots\lead{g_N}|$.
\end{lemma}

\begin{proof}
Fix $n\in \Z_+$. Recall that if $g=g_{i_n}\circ \dots \circ g_{i_1}$ then
\[
\lead{g}=\lead{g_{i_n}}{\lead{g_{i_{n-1}}}}^{d_{i_n}} \dots \lead{g_{i_1}}^{d_{i_2} 
\dots d_{i_n}},
\]
where $d_{i_j}= \deg (g_{i_j})$. Hence,
\begin{equation}\label{E:lead_g}
\prod_{l(g)=n}\!|\lead{g}|
= \prod_{(i_1,\dots, i_n)\,\in\,\{1, \dots, N\}^n}\!\!|\lead{g_{i_n}}{\lead{g_{i_{n-1}}}}^{d_{i_n}}\dots \lead{g_{i_1}}^{d_{i_2}
\dots d_{i_n}}|.
\end{equation}
If we fix a word $(g_{i_n}\circ\dots \circ g_{i_{k+1}})$, $1\leq k\leq n$, and
an $i\in \{1,\dots, N\}$, then the number of words of the form
$g_{i_n}\circ\dots \circ g_{i_{k+1}}\circ g_i\circ f$\,---\,where
\[
\begin{split}
g_{i_n}\circ\dots \circ g_{i_{k+1}}:=&\,{\sf id}_{\sph} \; \text{ if $k =n$, and} \\ 
f =&\,\begin{cases}
	\text{a word with $l(f) = (k-1)$}, &\text{if $k\geq 2$}, \\
	{\sf id}_{\sph}, &\text{if $k=1$}
	\end{cases}
\end{split}
\]
---\,that contribute a factor of $\lead{g_i}^{d_{i_{k+1}}\dots d_{i_n}}$ to the 
right-hand side of \eqref{E:lead_g} is $N^{k-1}$.
Here, we shall set $d_{i_{k+1}}\dots d_{i_n}:= 1$ if $k=n$.
Thus the right-hand side of \eqref{E:lead_g} can be reorganized as follows:
\[
\begin{split}
  \prod_{l(g)=n}|\lead{g}| &=
  \prod_{i=1}^N\Bigg( |\lead{g_i}|^{N^{n-1}}\times\prod_{k=1}^{n-1}
     \Bigg(\prod_{(i_{k+1},\dots, i_n)\,\in\,\{1,\dots,N\}^{n-k}}
   \Big({|\lead{g_i}|}^{d_{i_{k+1}}\dots d_{i_n}} \Big)^{N^{k-1}}\Bigg)\Bigg)\\
  &=\prod_{i=1}^N\,\prod_{k=1}^n 
  \Big(|\lead{g_i}|^{D^{n-k}}\Big)^{N^{k-1}}\\
  &=\prod_{i=1}^N 
  |\lead{g_i}|^{(D^{n-1}+ND^{n-2}+\dots+DN^{n-2}+N^{n-1})},
\end{split}
\]
since $d_1+d_2+\dots +d_N = D$. For simplicity of notation, let us write
\[
A:=|\lead{g_1}\lead{g_2}\dots\lead{g_N}|.
\]
Then, it follows from above that
\[
\begin{split}
\frac{1}{D^n} \log\Bigg(\prod_{l(g)=n}|\lead{g}|\Bigg) 
&=\frac{1}{D^n} \log \left(A^{(D^{n-1}+ND^{n-2}+\dots+DN^{n-2}+N^{n-1})}\right)\\
&=\frac{1-{(N/D)}^n}{D-N}\,\log A.
\end{split}
\]
Since $N<D$ by assumption, $(N/D)^n \to 0$ as $n\to \infty$, from which the result follows.
\end{proof}

Before proving Theorem~\ref{T:main}, we formally define a term that was used in 
Section~\ref{S:intro}.

\begin{definition}\label{D:regular}
Let $X$ be a topological space, and let $u: X \to [-\infty, \infty)$ be a 
function that is locally bounded above on $X$. Its \emph{upper semicontinuous regularization} 
$u^*:X \to [-\infty, \infty)$ is defined by
\[
u^*(x):=\limsup_{y \to x}u(y)=\inf_{\mathscr{N}}
(\sup_{y \in \mathscr{N}}u(y)) \quad \forall x\in X,
\]
the infimum being taken over all neighbourhoods $\mathscr{N}$ of $x$.
\end{definition}

It is easily checked that $u^*$ is an upper semicontinuous function on $X$ such that 
$u^* \geq u$, and also that it is the least upper semicontinuous function that dominates $u$.
\smallskip

\begin{proof}[The proof of Theorem~\ref{T:main}]
Note that if $g$ is a polynomial then
\[
g(z)-a=\lead{g} \prod_{j=1}^{\deg (g)}(z-x_j),
\]
where $x_1,\dots,x_{\deg (g)}$ are the solutions of $g(z)=a$ repeated according to multiplicity.
Now fix a set of generators $\gen=\{ g_1, g_2,\dots,g_N\}$.
Fix $\fixed \in \C \setminus E(\gen)$, where $E(\gen)$ as in the proof of Theorem~\ref{T:conti}, and
let $\mu^{\fixed}_n$ be as defined in the first paragraph of that proof.
Set $\mu_n:= \mu^{\fixed}_n$.
Then, by the definition of the logarithmic potential,
\[
\begin{split}
U^{\mu_n}(z) & =\frac{1}{D^n} 
\displaystyle \sum_{l(g)=n}\; \sum_{x \in {g^{-1}\{\fixed\}}^{\bullet}}\!\! \log \left (\frac{1}{|z-x|}\right)\\
& =\frac{1}{D^n} \log \left (\displaystyle \prod_{l(g)=n}\frac{|\lead{g}|}{|g(z)-\fixed|}\right).
\end{split}
\]
Therefore, by Lemma~\ref{L:compu},
\[
\begin{split}
\liminf_{n \to \infty} U^{\mu_n }(z)
 & =\liminf_{n \to \infty} \frac{1}{D^n} 
 \log \left (\displaystyle \prod_{l(g)=n}\frac{|\lead{g}|}{|g(z)-\fixed|}\right)\\
 & = \liminf_{n \to \infty} \frac{1}{D^n} 
 \log \left (\displaystyle \prod_{l(g)=n}\frac{1}{|g(z)-\fixed|}\right) 
+\lim_{n \to \infty} \frac{1}{D^n} \log \left (\displaystyle \prod_{l(g)=n}|\lead{g}|\right)\\
 & = -G_{\gen}(z)+ \frac{\log A}{D-N},
\end{split}
\]
where $A=|\lead{g_1}\lead{g_2}\dots\lead{g_N}|$.
\smallskip

Since, by \eqref{E:asymp2}, $\mu_n \to \mu_{\gen}$ in the weak* topology, Result~\ref{R:PDLE} implies:
\[
\begin{split}
U^{\mu_{\gen}}(z) &\leq  -G_{\gen}(z)+ \frac{\log A}{D-N} \qquad \text{ for every } z \in \C,\\
U^{\mu_{\gen}}(z) &=-G_{\gen}(z)+ \frac{\log A}{D-N} \qquad \text{ for q.e. } z \in \C.
\end{split}
\]
By Theorem~\ref{T:conti}, $U^{\mu_{\gen}}$ is continuous. Thus $G_{\gen}$ is locally bounded above.
Hence, $G^*_{\gen}$ (the upper semicontinuous regularization of $G_{\gen}$) exists.
Since $G^*_{\gen}$ is the least upper semicontinuous function that dominates $G_{\gen}$
and $U^{\mu_{\gen}}$ is continuous, we get
\begin{align}
U^{\mu_{\gen}}(z) 
&\leq -G^*_{\gen}(z)+ \frac{\log A}{D-N} \qquad \text{ for every } z \in \C,\label{E:LE1}\\
U^{\mu_{\gen}}(z) 
&=- G^*_{\gen}(z)+ \frac{\log A}{D-N} \qquad \text{ for q.e. } z \in \C.\label{E:LE2}
\end{align}
Now, if $U^{\mu_{\gen}}$ is continuous at some $z_0 \in \C$ and
\[
U^{\mu_{\gen}}(z_0) < - G^*_{\gen}(z_0)+ \frac{\log A}{D-N}
\]
then by lower semicontinuity of $-G^*_{\gen}$, we can find an open disc $\Delta$ with centre $z_0$ such that
for every $z \in \Delta$,
\[
U^{\mu_{\gen}}(z) < - G^*_{\gen}(z)+ \frac{\log A}{D-N}.
\]
Since $\cpt(\Delta) > 0$, the last inequality contradicts \eqref{E:LE2}. Thus
\[
U^{\mu_{\gen}}(z_0) =- G^*_{\gen}(z_0)+ \frac{\log A}{D-N}.
\]
As $U^{\mu_{\gen}}$ is continuous, the above argument implies that
\begin{equation}\label{E:equality_on_C}
U^{\mu_{\gen}}(z)=- G^*_{\gen}(z)+ \frac{\log A}{D-N} \quad \forall z \in \C.
\end{equation}
Therefore, $G^*_{\gen}$ is continuous on $\C$. For the remainder of our argument we shall 
draw upon results in Section~\ref{S:potential}. We introduce some notation in order to
identify objects featuring in our proof with those in
Section~\ref{S:potential}. We define $\Q:=G^*_{\gen}|_{\J(S)}$. Then
\begin{equation}\label{E:frost}
U^{\mu_{\gen}}(z) + \Q(z) = \frac{\log A}{D-N} \quad \forall z \in \J(S).
\end{equation}
In the notation of Section~\ref{S:potential}, consider the closed set $\Sigma:= \J(S)$.
Now observe:
\renewcommand{\theenumi}{{\emph{\roman{enumi}}}}
  \begin{enumerate}[leftmargin=27pt]
\item $\Q$ is continuous on $\Sigma$;
\item\label{I:finite} $\Q(z)$ is finite for every $z \in \Sigma$.
\end{enumerate}
As seen earlier, since $S$ contains a polynomial of degree at least 2,
\[
\cpt(\Sigma)=\cpt(\J(S))>0.
\]
By observation (\ref{I:finite}) above and 
since $\Sigma$ is of positive capacity, the set on which $\Q < \infty$ is of positive capacity. 
Since $\Q$ is continuous on $\Sigma$ and $\J(S)$ is compact, 
it follows that $\Q$ is an external field (or equivalently, 
$\wgen(z):= \exp(-\Q(z))$ is an admissible weight).
\smallskip

Since $\supp(\mu_{\gen}) \subset \J(S)$, $\mu_{\gen}$ is compactly supported in $\C$.
It is easy to see that, since $U^{\mu_{\gen}}$ is continuous on $\C$, 
$\mu_{\gen}$ has finite logarithmic energy.
By \eqref{E:frost} and Result~\ref{R:frost}, $\mu_{\gen}$ is the weighted equilibrium 
measure associated with the external field $\Q$ and 
\begin{equation}\label{E:robin}
F_{\gen}:=F_{\wgen}=\frac{1}{D-N}\log |\lead{g_1}\lead{g_2}\dots\lead{g_N}|
\end{equation}
is the modified Robin constant for $\Q$.
\end{proof}

\begin{remark}\label{Rem:independence}
We point out that the last proof reveals that the function $G^*_{\gen}$ does not
depend on the choice of $a \in \C$, provided $a\notin E_{\gen}$. With this exception, we know from
\eqref{E:asymp2} that $\mu_{\gen}$, and hence $U^{\mu_{\gen}}$, is independent of $a$.
Thus, in view of \eqref{E:equality_on_C}, the stated independence of the choice of $a$ follows.
We shall exploit this fact in Section~\ref{S:bound}, where we shall work with $G^*_{\gen}$ and
$\Q$.
\end{remark}

\section{The proof of Theorem~\ref{T:bound}}\label{S:bound}

Before giving a proof of Theorem~\ref{T:bound}, we state couple of results that we will need.

\begin{result}[Pommerenke, \cite{pommerenke:upsPm79}]\label{R:unif_per_cap}
Let $\Sigma$ be a non-empty closed subset of $\C$. Then, 
$\Sigma$ is uniformly perfect if and only if there is a constant $\delta >0$ such that
$\cpt(\Sigma \cap \overline{D(z,r)})\geq \delta r$ for all $z \in \Sigma$ and $0<r< {\rm diam}(\Sigma)$
(where ${\rm diam}(\Sigma)$ denotes the diameter of $\Sigma$).
\end{result}

\begin{remark}
The above result was stated in \cite{pommerenke:upsPm79} for unbounded closed sets in $\C$. 
However, the only place where the unboundedness of $\Sigma$ is needed in its proof is in taking
$\Sigma$ to have the property stated in Remark~\ref{Rem:unif_perf_planar}, which is almost
immediate when $\Sigma$ is unbounded.
\end{remark}

The measure $\mu_{\gen}$ in the next result is as in the previous sections.

\begin{result}[Boyd, \cite{boyd:imfgrs99}]\label{R:boyd_support}
Let $S$ be a finitely generated rational semigroup where each element of $S$ has degree at 
least 2. Let $\gen$ be a finite set of generators. Then
$\supp(\mu_{\gen}) = \J(S)$.
\end{result}

The exterior boundary of $\J(S)$ will be denoted by by $\partial_e\J(S)$.
Also, we recall that owing to Theorem~\ref{T:conti},
\eqref{E:equality_on_C} tells us that $G^*_{\gen}$, and hence $\Q$, are continuous. We
shall use this without any further comment in%

\begin{proof}[The proof of Theorem~\ref{T:bound}]
We begin with the proof of part~\ref{I:ext_field}. Fix a finite set of generators $\gen$.
By hypothesis, the orbit of $z_0$, $O(z_0)$, is not dense in $\C$. Thus there exist $p \in \C$ and $\eps>0$
such that $O(z_0) \cap D(p, 2\eps)= \emptyset$.
As $\cpt (D(p, \eps)) >0$, there is a point $\fixed \in D(p, \eps)$ such that $\fixed \notin  E(\gen)$.
Recall that, $G^*_{\gen}$ is the upper 
semicontinuous regularization of the function
\[
G_{\gen}(z)=\limsup_{n \to \infty} \frac{1}{D^n} 
\log \left (\prod_{l(g)=n} {|g(z)-\fixed|}\right)
\]
and the external field $\Q=G^*_{\gen}|_{\J(S)}$. Note that, by Remark~\ref{Rem:independence},
$\Q$ does not depend on the choice of $\fixed$, where $\fixed \in \C \setminus E(\gen)$.
Let $\rho_1 >0$ be such that
\[
|g_i(z)|> |z|,\quad |g_i(z)-\alpha|>\frac {|\Lambda(g_i)|}{2}|z|^{\deg (g_i)}
\quad \forall z: |z|>\rho_1,
\]
for $i=1,\dots,N$ and $\alpha=0, \fixed$.
Then, owing to Lemma~\ref{L:compu} and the above choice of $\rho_1$, we get
\[
G_{\gen}(z) \geq \log|z|+\frac{1}{D-N}\big(\log |\lead{g_1}\lead{g_2}\dots\lead{g_N}|
-N\log2\big), \quad \forall z:|z|>\rho_1.
\]
Hence, there exists a $\rho_2\geq \rho_1$ such that $G_{\gen}(z)>0$ whenever $|z|>\rho_2$.
Since $O(z_0)$ is unbounded, there exists a word $h$ such that $|h(z_0)|>\rho_2$, whence
$G_{\gen}(h(z_0))>0$. Let $l(h)=M$. 
Observe that by the choice of $\fixed$, for any $g\in S$, $|g(z_0)-\fixed|> \eps$. Thus, for $n \geq M+1$,
\[
\log \left (\prod_{l(g)=n} {|g(z_0)-\fixed|}\right)
> \sum_{l(g)=n-M}\log{|g(h(z_0))-\fixed|}+(N^n-N^{(n-M)}) \log \eps.
\]
Divide both sides above by $D^n$. Then, it follows from the definition of $G_{\gen}$, 
and as $N/D< 1$, that
$G_{\gen}(z_0) \geq D^{-M}  G_{\gen}(h(z_0))$.
Recall that $G_{\gen}(h(z_0))>0$. It follows that
$\Q(z_0)= G^*_{\gen}(z_0) \geq G_{\gen}(z_0) >0$. Since the choice of $\gen$
was arbitrary, this establishes \ref{I:ext_field}.
\smallskip

We shall prove part~\ref{I:nec_cond} by establishing its contrapositive.
Assume that for every $z \in \J(S)$, $O(z)$ is bounded. Now fix a finite set of
generators $\gen$. By Lemma~\ref{L:compact},
there exists $R >0$ (which depends on $\gen$) such that $O(z)\subset D(0,R)$
for every $z \in \J(S)$. Choose $\fixed \in \C$
such that $|\fixed| > 2R$ and $\fixed \notin E(\gen)$. Observe:
\[
 N^n \log R\leq \log \left (\prod_{l(g)=n} {|g(z)-\fixed|}\right) \leq N^n \log (R+|a|)
 \quad \forall z\in \J(S).
\]
Since $N/D < 1$, it follows that $G_{\gen}(z)=0$ for every $z \in \J(S)$.
By \eqref{E:LE2},
$G_{\gen}=G^*_{\gen}$ quasi-everywhere on $\C$.
In particular, $\Q = 0$ quasi-everywhere on $\J(S)$. Suppose there is some
$\zeta\in \J(S)$ such that $\Q(\zeta)\neq 0$. Then, as $\Q$ is continuous, there exists a
disc $\Delta$ with centre $\zeta$ such that $\Q\neq 0$ on $\J(S)\cap \Delta$.
In view of Results~\ref{R:uniper} and~\ref{R:unif_per_cap}, $\J(S)\cap \Delta$ must have
positive capacity: a contradiction. Thus $\Q \equiv 0$. Since this
is true for any choice of $\gen$, \ref{I:nec_cond} follows.
\smallskip

It now remains to prove the capacity estimate in \eqref{E:bound_cap}.
To do so, we use the $F$-functional of Mhaskar and Saff with external field $\Q$ given by
Theorem~\ref{T:main}. We introduce some notation in order to identify
objects pertinent to our analysis with those in Section~\ref{S:potential}.
Abbreviate $\mathcal{S}_{\gen}:= \mathcal{S}_{w_{\gen}}$: the latter is as introduced
just prior to Result~\ref{R:frost}. 
By Theorem~\ref{T:main}, $\mathcal S_{\gen}=\supp(\mu_{\gen})$. With
this notation, Result~\ref{R:ms} gives
\[
\log \cpt(\mathcal S_{\gen}) = - F_{\gen} + \int_{\J(S)} \Q d\sigma_{\mathcal S_{\gen}}.
\]
By assumption, each element of $S$ is of degree at least 2. Thus, by Result~\ref{R:boyd_support}
and since $\mathcal S_{\gen}=\supp(\mu_{\gen})$, we have $\mathcal S_{\gen} =\J(S)$.
Also note that, by \eqref{E:robin}, the modified Robin constant is
$F_{\gen}=(D-N)^{-1}\log |\lead{g_1}\lead{g_2}\dots\lead{g_N}|$.
Consequently, we get
\begin{equation}\label{E:cbe}
\log \cpt(\J(S)) = \frac{1}{N-D}\log |\lead{g_1}\lead{g_2}\dots\lead{g_N}| 
+\int_{\J(S)} \Q d\equij.
\end{equation}

Let $V$ denote the unbounded component of $\F(S)$, whence 
$\partial_e \J(S):= \partial V$.
With $z_0$ as in our hypothesis, the proof of part~\ref{I:ext_field} tells us that
$G^*_{\gen}(z_0) >0$.
We claim that there exists
$z_1 \in \partial_e \J(S)$ such that $\Q(z_1) >0$.
If $z_0 \in \partial_e \J(S)$, take $z_1:= z_0$ and we are done. If 
$z_0 \notin \partial_e \J(S)$, then $\Omega:= \C \setminus (\partial_e \J(S)\cup V)$
is a non-empty open set. Recall that, by \eqref{E:equality_on_C}, $G^*_{\gen}$ is a continuous
subharmonic function on $\C$. Since $z_0\in \Omega$ and $G^*_{\gen}(z_0) > 0$,
by the maximum principle applied to
$\left. G^*_{\gen}\right|_{\Omega}$, the desired claim follows.
\smallskip

Let $x \in V \setminus \{\infty\}$. 
By Result~\ref{R:boyd04}, $x$ also belongs to the unbounded component of $\basin(S)$.
Thus, as argued just prior to Lemma~\ref{L:compact}, given $r>0$, there exists 
$n_r(x) \in \Z_+$ such that if $l(g)\geq n_r(x)$ then $|g(x)| > r$. 
It follows by taking $r>0$ large enough that $G_{\gen}(x)\geq 0$.
Since $G^*_{\gen} \geq G_{\gen}$, and $x\in V$ was arbitrary,
$G^*_{\gen}\geq 0$ on $V$. So, as each point in $\partial_e \J(S)$ is a limit point of $V$,
it follows from the continuity of $G^*_{\gen}$ that $G^*_{\gen}\geq 0$ on $\partial_e \J(S)$. Thus
$\Q\geq 0$ on $\partial_e \J(S)$.
\smallskip

We now appeal to potential theory to prove $\supp(\equij) = \partial_e \J(S)$.
It is well known that $\supp(\equij)\subseteq \partial_e\J(S)$.
Let $\zeta \in \partial_e \J(S)$ and assume that $\zeta \notin \supp(\equij)$. Then,
there exists an open disc $\Delta$ with centre $\zeta$ on which
$U^{\sigma_{\J(S)}}$ is harmonic. By Result~\ref{R:uniper}, $\J(S)$ is uniformly perfect. Thus, by
Result~\ref{R:unif_per_cap}, it follows that $V$ is regular\,---\,see, for 
instance \cite[Corollary~2 to Theorem~III-62]{tsuji:ptmft75}.
As $\zt$ is a regular boundary point of $V$, and $\infty\in V$, it is a classical fact that
$U^{\sigma_{\J(S)}}(\zt)= -\log \cpt (\J(S))$. But by Frostman's theorem,
$U^{\sigma_{\J(S)}}\leq -\log \cpt (\J(S))$ on $\C$. Thus, applying the maximum principle to
$\left. U^{\sigma_{\J(S)}}\right|_{\Delta}$, we have $\left. U^{\sigma_{\J(S)}}\right|_{\Delta}\equiv
-\log \cpt (\J(S))$. As $U^{\sigma_{\J(S)}}$ is harmonic on $V\setminus \{\infty\}$, the identity
principle for harmonic functions implies that $U^{\sigma_{\J(S)}}\equiv -\log \cpt (\J(S))$ on 
$V \setminus \{\infty\}$. This contradicts the fact that $U^{\sigma_{\J(S)}}(z)= -\log|z|+ o(1)$
as $z \to \infty$. Hence, $\supp(\equij) = \partial_e \J(S)$. Thus, by the continuity
of $\Q$, there is a $\J(S)$-open neighbourhood $\mathscr{N}$ of $z_1$ with $\equij(\mathscr{N}) > 0$
such that $\Q > 0$ on $\mathscr{N}$. This, together with the conclusions of the last two paragraphs,
gives
\[
\int_{\J(S)} \Q d\equij > 0.
\]
By \eqref{E:cbe}, we get the desired inequality.
\end{proof}

\begin{corollary}
Let $S$ be a finitely generated polynomial semigroup as in Theorem~\ref{T:bound}
and let $\gen=\{ g_1, g_2,\dots,g_N\}$ 
be a set of generators of $S$. 
If each element of $S$ is of degree at least 2 then
\[
{{\rm diam}}(\J(S)) > 2{|\lead{g_1}\lead{g_2}\dots\lead{g_N}|}^{\frac{1}{N-D}},
\]
where ${{\rm diam}}(\J(S))$ denotes the diameter of $\J(S)$ with respect to the Euclidean 
metric on $\C$.
\end{corollary}
\begin{proof}
The estimate is a consequence of the following relation between
logarithmic capacity and diameter: if $K$ is a compact subset of $\C$ then
\[
\cpt(K) \leq \frac{{{\rm diam}}(K)}{2}
\]
---\,see, for instance, \cite[Theorem~5.3.4]{ransford:ptCp95}. The result now follows from 
\eqref{E:bound_cap}.
\end{proof}

\section*{Acknowledgments}
\noindent{I would like to thank my thesis adviser, Prof. Gautam Bharali, for several fruitful discussions
and  for his help with the writing of this paper. I also thank the referee for his/her helpful suggestions.
This work is supported by a scholarship from the National Board for Higher Mathematics
(Ref. No. 2/39(2)/2016/NBHM/R\&D-II/11411) and by a UGC CAS-II grant (Grant No. F.510/25/CAS-II/2018(SAP-I)).}


\medskip


\begin{thebibliography}{88}

\bibitem{arsove:cplmd60}
Maynard G. Arsove, {\em Continuous potentials and linear mass distributions},
SIAM Rev. {\bf 2} (1960), no.\,3,
177--184.

\bibitem{BhaSri:hcrfgrs17}
Gautam Bharali and Shrihari Sridharan, {\em Holomorphic correspondences related to finitely generated rational semigroups},
Internat. J. Math. {\bf 28} (2017), no.\,14,
25 pp.

\bibitem{boyd:imfgrs99}
David Boyd, {\em An invariant measure for finitely generated rational semigroups}, Complex Variables Theory Appl.
{\bf 39} (1999), no.\,3,
229--254.

\bibitem{boyd:ibaipsft04}
David A. Boyd, {\em The immediate basin of attraction of infinity for polynomial semigroups of finite type},
J. London Math. Soc. (2) {\bf 69} (2004), no.\,1,
201--213.

\bibitem{brolin:isuirm65}
Hans Brolin, {\em Invariant sets under iteration of rational functions},
Ark. Mat. {\bf 6} (1965), 
103--144.

\bibitem{DinhKaufWu:prmdpv20}
Tien-Cuong Dinh, Lucas Kaufmann, Hao Wu, {\em Products of random matrices: a dynamical point of view}, 
to appear in Pure Appl. Math. Q., arXiv reference: \texttt{arXiv:1905.08461}.


\bibitem{DinhSibony:emtddc05}
Tien-Cuong Dinh and Nessim Sibony, {\em Equidistribution for meromorphic transforms and the
$dd^{{\rm c}}$-method}, Sci. China Ser. A {\bf 48} (2005), suppl., 180--194.

\bibitem{DinhSibony:dvtma06}
Tien-Cuong Dinh and Nessim Sibony, {\em Distribution des valeurs de transformations m{\'e}romorphes et applications},
Comment. Math. Helv. {\bf 81} (2006), no.\,1, 221--258.

\bibitem{frostman:potential35}
O. Frostman, {\em Potential d'{\'e}quilibre et capacit{\'e} des ensembles avec quelques applications {\`a} la 
th{\'e}orie des fonctions}, Lunds Univ. Mat. Sem. {\bf 3} (1935), 
3--115.

\bibitem{HinkMart:dsrf96}
A. Hinkkanen and G. J. Martin, {\em The dynamics of semigroups of rational functions I},
Proc. London Math. Soc. (3) {\bf 73} (1996), no.\,2,
358--384.

\bibitem{HinkMart:jsrs96}
A. Hinkkanen and G. J. Martin, {\em Julia sets of rational semigroups},
Math. Z. {\bf 222} (1996), 
161--169.

\bibitem{pommerenke:upsPm79}
Ch. Pommerenke, {\em Uniformly perfect sets and the Poincar{\'e} metric}, Arch. Math. {\bf 32} (1979),
no.~2, 192--199.

\bibitem{pommerenke:upsFg84}
Ch. Pommerenke, {\em On uniformly perfect sets and Fuchsian groups}, Analysis {\bf 4}
(1984), no.~3-4, 29--321.

\bibitem{ransford:ptCp95}
Thomas Ransford, {\em Potential Theory in the Complex Plane}, LMS Student Texts {\bf 25},
Cambridge University Press, Cambridge (UK), 1995.

\bibitem{SaffTotik:lpwef97}
Edward B. Saff and Vilmos Totik, {\em Logarithmic Potentials with External Fields},
Springer-Verlag, Berlin, 1997.

\bibitem{stankewitz:upsrsKg2000}
Rich Stankewitz, {\em Uniformly perfect sets, rational semigroups, Kleinian groups and IFS's},
Proc. Amer. Math. Soc. {\bf 128} (2000), no.~9, 2569--2575.

\bibitem{sumi:spmrfgrs00}
Hiroki Sumi, {\em Skew product maps related to finitely generated rational semigroups}, Nonlinearity
{\bf 13} (2000), no.\,4, 995--1019.

\bibitem{tsuji:ptmft75}
M. Tsuji, {\em Potential Theory in Modern Function Theory}, 2nd edition, Chelsea Publishing Co., New York (NY), 1975.

\end{thebibliography}
\end{document}